\documentclass[A4,11pt,reqno]{amsart}
\usepackage{indentfirst,enumerate}
\usepackage{amssymb,amsfonts,amsmath,amsthm,mathrsfs}
\usepackage{color}
\usepackage{hyperref}
\usepackage{multicol}

\newtheorem{theorem}{Theorem}[section]

\newtheorem{proposition}[theorem]{Proposition}
\newtheorem{lemma}[theorem]{Lemma}
\newtheorem{remark}[theorem]{Remark}
\newtheorem{example}[theorem]{Example}
\newtheorem{definition}[theorem]{Definition}

\numberwithin{equation}{section}

\newcommand{\T}{\mathbb{T}}
\newcommand{\R}{\mathbb{R}}
\newcommand{\Z}{\mathbb{Z}}
\newcommand{\N}{\mathbb{N}}
\newcommand{\C}{\mathbb{C}}

\newcommand{\Rea}{{\rm Re\, }}
\newcommand{\Ima}{{\rm Im\, }}
\newcommand{\spec}{\textit{spec}\,}
\newcommand{\D}{\mathcal{D}'}
\renewcommand{\H}{\mathcal{H}}
\renewcommand{\to}{\rightarrow}

\begin{document}
	
	\title[Perturbations of (GH) operators on closed manifolds]{Perturbations of globally hypoelliptic \\ operators  on closed manifolds}
	
	\author[Fernando de \'{A}vila]{Fernando de \'{A}vila Silva}
	\address{
		Fernando de \'{A}vila Silva:
		Departamento de Matem\'{a}tica
		Universidade Federal do Paran\'{a},
		Caixa Postal 19081, Curitiba, PR 81531-990,
		Brazil
	}
	
	\email{fernando.avila@ufpr.br}
	
	\author[Alexandre Kirilov]{Alexandre Kirilov }
	\address{
		Alexandre Kirilov:
		Departamento de Matem\'{a}tica,
		Universidade Federal do Paran\'{a},
		Caixa Postal 19081, Curitiba, PR 81531-990,
		Brazil
	}
	
	\email[Corresponding author]{akirilov@ufpr.br}
	
	

\begin{abstract}
	Analyzing the behavior at infinity of the sequence of eigenvalues given by matrix symbol of a invariant operator with respect to a fixed elliptic operator, we obtain necessary and sufficient conditions to ensure that perturbations of globally hypoelliptic operators continue to have this property. As an application, we recover classical results about perturbations of constant vector fields on the torus and extend them  for more general classes of perturbations. Additionally, we construct examples of low order perturbations that destroy the global hypoellipticity, in the presence of diophantine phenomena.
		
\end{abstract}
	
\maketitle
	
\tableofcontents

\section{Introduction}

In this note, we are concerned with perturbations of globally hypoelliptic operators of the form
\begin{equation}\label{L-dif}
L = D_t +  Q, \ (t,x) \in \T \times M,
\end{equation}
where $D_t=-i\partial/\partial t$, $\T \simeq \R / 2 \pi \Z$ is the flat torus, $M$ is a closed smooth manifold and $Q$ is a continuous linear operator on $\D(M)$. We also assume that $Q$ commutes with an elliptic operator defined on $M$ and that the domain of the adjoint operator ${Q}^*$ contains $C^\infty(M)$.

The assumption of commutativity with an elliptical operator $E$ introduces a Fourier analysis on the manifold $M$, whereas the hypothesis on the domain of the adjoint operator ensures  that the Fourier coefficients $\widehat{Qu}(\ell)= \sigma_{\!_Q}(\ell)\widehat{u}(\ell)$, for $u \in C^\infty(M)$, where $\sigma_{\!_Q}$ is the matrix symbol of the operator $Q$. For more details, see Section 4 of \cite{DR-Fmul}.

We recall that an operator $L$ is globally hypoelliptic (GH) on $\T\times{M}$ if the conditions $u \in \D(\T \times M)$ and $L u \in C^{\infty}(\T \times M)$ imply $u \in C^{\infty}(\T \times M)$.

The global hypoellipticity of \eqref{L-dif}, in the case where $Q$ is a first order normal pseudo-differential operator, that commutes with an elliptic operator $E$, was dealt in \cite{AvGrKi}. In that article, it was proved that the global hypoellipticity of this operator is related to the behavior, at infinity, of sequences of eigenvalues given by the matrix symbol of $Q$.

Following the approach of \cite{AvGrKi}, we study the global hypoellipticity of the operator \eqref{L-dif} and of the perturbed operator
\begin{equation*}
L(\epsilon) = D_t +  Q + \epsilon \mathcal{R}, \ (t,x) \in \T \times M,
\end{equation*}	
by analyzing the behavior of the sequence of eigenvalues of the restrictions
\begin{equation*}
Q_j \doteq Q\Big|_{E_{\lambda_j}} \textrm{ and \ } \ Q_{j}(\epsilon)  \doteq ( Q + \epsilon \mathcal{R} )\Big|_{E_{\lambda_j}},
\end{equation*}	
where $E_{\lambda_j}$ are the eigenspaces of $E$ and  $\epsilon$ is assumed to be small. Here, the perturbations $\mathcal{R}$ are also continuous linear operators on $\D(M)$ that commute with $E$, and the domain of $\mathcal{R}^*$ contains $C^\infty(M)$.

One of the inspirations for this note was the article \cite{BERG94} of A. Bergamasco, in which the author characterizes the global hypoellipticity of a class of perturbed operators defined on the torus by a Diophantine condition on the symbol of the operators. This type of condition was observed first by S. Greenfield and N. Wallach \cite{GW1}, in the case of constant coefficient operators, and by several other authors in the study of global hypoellipticity of operators on tori, see \cite{BERG99,BCM,BDGK15,CC00,GPY1,HOU82,Petr06,Petr11}.

In Theorem \ref{general-GH}, assuming that $Q$ is strongly diagonalizable, see Definition \ref{strong-diag}, we exhibit sufficient and necessary conditions to the global hypoellipticity of $L=D_t+Q$.

In section \ref{section4}, exploring the characterization of the strongly invariant operators and the strength of Theorem \ref{general-GH}, we present a class of perturbations on the torus, invariant with respect to Laplacian, and we obtain new results regarding perturbations of constant vector fields.  In particular, we recover results of Bergamasco \cite{BERG94} and extended them to a wider class of perturbations. For example, when $L=D_t +  \alpha D_x$ and $\alpha$ is an irrational number, we construct perturbations $\mathcal{R}$, of any order less than 1, such that $D_t +  \alpha D_x + \mathcal{R},$ is not (GH), see Theorem \ref{existence-nGH}.

In section \ref{section5}, motivated by T. Kato's  and F. Rellich's books, see \cite{kato} and \cite{Rellich}, we assume that $Q$ is normal and the eigenvalues and eigenvectors of $Q_{j}(\epsilon)= Q_j + \epsilon\mathcal{R}_j$ have analytic expansions is series of powers. Approaching the problem from this point of view, we have an algorithmic method for calculating (at least approximately) the eigenvalues of the operator $Q(\epsilon)$.

Finally, in the last section, we recover and extend results of section \ref{section4}, explicitly presenting the calculations of eigenvalues and eigenvectors of the perturbed operators and analyzing the global hypoellipticity.

\section{Fourier analysis relative to an elliptic operator}

Let $\N_0=\N\cup\{0\}$, $\langle \cdot, \cdot \rangle_{\C^d}$ be the usual inner product of $\C^d$, and $M$ be a n-dimensional closed smooth manifold endowed with a positive measure $dx$. Consider the space $L^2(M)$ of square integrable complex-valued functions on $M$, with respect to $dx$, with the inner product
\begin{equation*}
(f,g)_{L^2(M)} \doteq \int_{M}{f(x)\overline{g(x)}dx}, \ \ f,g \in L^2(M).
\end{equation*}

We denote by $\H^s(M)$ the standard Sobolev space of order $s$ on $M$, thus
\begin{equation*}\label{sobolev-dx}
C^\infty(M) = \bigcap_{s \in \R}\H^{s}(M) \mbox{ \ and \ } \D(M) = \bigcup_{s \in \R}\H^{s}(M).
\end{equation*}

Let $\Psi^m_{+e}(M)$ be the class of the classical positive elliptic pseudo-differential operators, of order $m\in \R$, and $E \in \Psi^m_{+e}(M)$ be a fixed elliptic operator. Following the construction proposed by Delgado and Ruzhansky, we introduce a discrete Fourier analysis in $M$ associated to $E$. Moreover, we assume that $m> 0$, which allows us to use the formula of Weyl to estimate the asymptotic behavior of the eigenvalues of $E$ (see \cite{Shubin}, sections 15$-$16), and characterize the Sobolev spaces in terms of this Fourier expansion (see \cite{GPR}). In this way:
\begin{enumerate}[{\it i.}]
	\item the spectrum $\spec(E)$ is a discrete subset of $\R$ and coincides with the set of all its eigenvalues.  Thus, the eigenvalues of E, counting the multiplicity, form a sequence
	$$ 0=\lambda_0 < \lambda_1 \leqslant \ldots \leqslant \lambda_j \longrightarrow \infty;$$
	
	\item from asymptotic formula of Weyl, there is  $c>0$, depending on $E$ and $M$, such that
	$$\lambda_j \sim \ c j^{\,{m}/{n}}, \  \mbox{ as }  j \to \infty;$$
	
	\item for each $j\in \N_0$, the eigenspace $E_{\lambda_j}$ of $E$ is a finite dimensional subspace of $C^\infty(M)$, and we will denote
	$$d_j\doteq \dim\, E_{\lambda_j};$$
	
	\item there is an orthonormal basis $\{e_{j}^{k}\ ; {1 \leqslant k \leqslant d_j} \mbox{ and } {j \in \N_0} \}$ for $L^2(M)$, consisting of smooth eigenfunctions of $E$ such that, for each $j\in \N_0$, $\{e_{j}^{1},e_{j}^{2},\ldots, e_{j}^{d_j}\}$ is an orthonormal basis of $E_{\lambda_j},$ and  $$L^2(M) = \bigoplus_{j\in \N_0} \, E_{\lambda_j};$$
	
	\item the Fourier coefficients of a function $f \in L^2(M)$, with respect to this orthonormal basis, are given by
	$$ \widehat{f}_j^{\, k}\doteq \big( f \, , \, e_{j}^{k} \big)_{L^2(M)}, \, 1\leqslant k \leqslant d_j, \  j \in \N_0.$$
	We also write $\widehat{f}_j = \big( \widehat{f}_j^{\, 1}, \ldots, \widehat{f}_j^{\, d_j}\big), \ j \in \N_0$; \smallskip
	
	\item any distribution $u \in \D(M)$ can be represented by its Fourier series
	\begin{equation*}\label{Fourier1}
	u = \sum_{j \in \N_0} \sum_{k=1}^{d_j} \widehat{u}_j^k e_j^k(x) = \sum_{j\in \N_0}\left\langle  \widehat{u}_j, \overline{e_j}(x) \right\rangle_{\C^{d_j}},
	\end{equation*}
	where $\widehat{u}_j^k = u(\overline{e_j^k}(x))$ and $\widehat{u}_j = \big( \widehat{u}_j^{\, 1}, \ldots, \widehat{u}_j^{\, d_j}\big), \ j \in \N_0$;
	
	\item for a distribution $u \in \D(M)$ we have
	\begin{equation} \label{Sobolev1matr}
	u \in \H^s(M) \Leftrightarrow
	\sum_{j\in\N_0} {\|\widehat{u}_j\|^2_{\C^{d_j}} \lambda_j^{\frac{2s}{m}} }< + \infty \Leftrightarrow
	\sum_{j\in\N_0} {\|\widehat{u}_j\|^2_{\C^{d_j}} j^{\frac{2s}{n}} }< + \infty.
	\end{equation}
\end{enumerate}

\begin{proposition}\label{prop-smooth-1}
	The three following statements on the series
	\begin{equation}\label{1-prop-smooth-1}
	\sum_{j \in \N_0} \sum_{k=1}^{d_j} c_j^{k} e_j^k(x),
	\end{equation}
	with complex coefficients $c_j^k$, are equivalent:
	\begin{enumerate}
		\item[i.] The series (\ref{1-prop-smooth-1}) converges in the $C^{\infty}(M)$ topology;
		\item[ii.] The series (\ref{1-prop-smooth-1}) is the Fourier expansion of some $f\in C^{\infty}(M)$. Furthermore $\widehat{f}_j^{\, k} = c_j^{k}$, for any $1\leqslant k \leqslant d_j$ and $j \in \N_0$;
		\item[iii.] For any integer $N$ we have
		\begin{equation}\label{2-prop-smooth-1}
		\sum_{j \in \N_0} \|c_j\|^2_{\C^{d_j}} j^{-N} < + \infty.
		\end{equation}
	\end{enumerate}
	Moreover, the following conditions are equivalent:
	\begin{enumerate}
		\item[iv.] The series (\ref{1-prop-smooth-1}) converges in the $\D(M)$ topology;
		\item[v.] The series (\ref{1-prop-smooth-1}) is the Fourier expansion of some $u\in \D(M)$.  Furthermore $\widehat{u}_j^{\, k} = c_j^{k}$, for any $1\leqslant k \leqslant d_j$ and $j \in \N_0$;
		\item[vi.] For some positive integer $N,$ (\ref{2-prop-smooth-1}) holds.
	\end{enumerate}
\end{proposition}

The  Fourier series with respect to variable $x$  of $u \in \D(\T \times M)$ is
\begin{equation}\label{x-fourier-series-distrib}
\sum_{j\in\N_0} \sum_{k=1}^{d_j} \widehat{u}_j^{\,k}(t) e_j^k(x) = \sum_{j\in\N_0} \big\langle \widehat{U}_j(t), \overline{e_j}(x) \big\rangle_{\C^{d_j}},
\end{equation}
where, for each  $j \in \N_0$, $\widehat{u}_j^{\,k}(t) = u\big(e_j^k(x)\big),$  ${e_j}(x)=\big(e_j^1(x), \ldots, e_j^{d_j}(x)\big)$, and
$\widehat{U}_j(t) = \big(\widehat{u}_j^1(t), \ldots, \widehat{u}_j^{d_j}(t)\big)$.

\smallskip
\begin{proposition}\label{prop-smooth-2}
	The three following statements on the series
	\begin{equation}\label{1-prop-smooth-2}
	\sum_{j \in \N_0} \sum_{k=1}^{d_j} c_j^k(t) e_j^k(x),
	\end{equation}
	with $c_j^k \in C^{\infty}(\T)$,  are equivalent:
	\begin{enumerate}[i.]
		\item the series converges in the $C^{\infty}(\T \times M)$ topology;
		\item the series is the $x$-Fourier expansion of some $f\in C^{\infty}(\T \times M)$. Furthermore $\widehat{f}_j^{\, k}(t) = c_j^{k}(t)$, for any $t \in \T, \ 1\leqslant k \leqslant d_j$ and $j \in \N_0$;
		\item for any $\alpha \in \N_0$ and integer $N$,
		\begin{equation} \label{2-prop-smooth-2}
		\max_{t\in \T}\|\partial_t^\alpha c_j(t)\|_{\C^{d_j}} = \mathcal{O}(j^{-N}), \textrm{ as } j \to \infty.
		\end{equation}
	\end{enumerate}
	Moreover, the following conditions on the series \eqref{1-prop-smooth-2} are equivalent:
	\begin{enumerate}
		\item[iv.] The series  converges in the $\D(\T \times M)$ topology;
		\item[v.] The series is the Fourier expansion of some $u\in\D(\T\times M)$.  Furthermore $\widehat{u}_j^{\, k}(t) = c_j^{k}(t)$, for any $t\in\T, \ 1\leqslant k \leqslant d_j$ and $j \in \N_0$;
		\item[vi.] For some integer $N$, \eqref{2-prop-smooth-2} holds.
	\end{enumerate}
\end{proposition}

The next result and definitions are a consequence of the results and remarks in Section 4 of \cite{DR-Fmul}.

\begin{proposition}\label{T-DR-1}
	Let ${\mathcal{A}}: \D(M) \to \D(M)$ be a continuous linear operator such that the domain of $\mathcal{A}^*$ contains $C^\infty(M)$. The following conditions are equivalent:
	\begin{enumerate}[i.]
		\item The operators $\mathcal{A}$ and $E$ commute, that is, $[\mathcal{A},E]=0$ on $\D(M)$;
		\item For each $j\in\N_0$, we have $\mathcal{A}(E_j)\subset E_j$;
		\item For each $j \in \N_0$, there is a matrix ${\mathcal{A}}_j \in \C^{d_j \times d_j}$ such that, for any $f \in C^{\infty}(M),$
		\begin{equation}\label{matrix-symbol}
		\widehat{({\mathcal{A}}  f)}_j = {\mathcal{A}}_j \widehat{f}_j.
		\end{equation}
		Moreover, ${\mathcal{A}}_j$ is the matrix of the restriction ${\mathcal{A}}\left|_{E_j}\right.$.
		
	\end{enumerate}
\end{proposition}

\begin{definition}\label{A-is-E-invariant}
	Let ${\mathcal{A}}: \D(M) \to \D(M)$ be a continuous linear operator such that the domain of $\mathcal{A}^*$ contains $C^\infty(M)$. We say that ${\mathcal{A}}$ is (strongly) invariant with respect to the operator $E$, or simply $E-$invariant, if it satisfies any of the equivalent conditions of Proposition \ref{T-DR-1}.
\end{definition}

\begin{definition}\label{order-matrix}
	Let $\mathcal{A}$ be an $E-$invariant operator, then:
	\begin{enumerate}[{\it i.}]
		\item  the family of matrices $\sigma_{{\mathcal{A}}} = \{\mathcal{A}_j\ \doteq (a_{k \ell})_{1\leqslant k, \ell \leqslant d_j}; \ j \in \N_0\}$, given by \eqref{matrix-symbol}, is called the  matrix symbol of ${\mathcal{A}}$  and
		\begin{equation*}
		{\mathcal{A}} f  (x) = \sum_{j \in \N_0} \left\langle {\mathcal{A}}_j^\top \widehat{f}_j , \overline{e_{\, j}}(x)  \right\rangle_{\C^2}, \ f \in C^{\infty}(M);
		\end{equation*}
		\item we say that the symbol $\sigma_{{\mathcal{A}}}$ has moderate growth if there are constants $C>0$, $N\in \R$ and $j_0\in\N$ such that
		\begin{equation}\label{moderate-symbol}
		\|\mathcal{A}_j\|  = \max_{1\leqslant k, \ell \leqslant d_j} |a_{k \ell}| \leqslant C j^{N},  \ j \geqslant j_0;
		\end{equation}
		
		\item if the symbol $\sigma_{{\mathcal{A}}}$ has moderate growth, then the order of $\mathcal{A}$ is the number
		$$ord(\mathcal{A}) \doteq     \inf \{ N \in \R ; \ \eqref{moderate-symbol} \ holds \}.$$
	\end{enumerate}
\end{definition}

The proofs of the results presented in this section can be found in \cite{DR14,DR-Fmul,DR14JFA}. For a version in the case where the basis of eigenfunctions is not orthogonal see \cite{RuzTok16IMRN}. Some applications of such analysis to spectral theory may be found in \cite{DelgRuzTok17}. Finally, for characterizations of other spaces, e.g. Gevrey spaces and ultradistributions, we refer the reader to \cite{DasRuz16}.

\section{Global hypoellipticity of invariant operators}

In this section we study the global hypoellipticity of the operator
\begin{equation*}
L = D_t + Q, \ (t, x) \in \T \times M,
\end{equation*}
where $Q$ is an $E-$invariant operator with matrix symbol $\{Q_j; \ j \in \N_0\}$ of moderate growth.

Firstly, for any $u \in \D(\T \times M)$, from \eqref{x-fourier-series-distrib} we have
\begin{equation*}\label{C}
Q u(t,x)  = \sum_{j\in \N_0} \sum_{k=1}^{d_j} u_j^k(t) \, Q e_j^k(x) = \sum_{j\in \N_0} \big\langle Q_j^\top \widehat{U}_j(t),  \overline{{e}_j}(x)  \big\rangle_{\C^{d_j}},
\end{equation*}
and, from the uniqueness of representation in Fourier series, the equation
\begin{equation*}\label{(D+Q)u=f}
\big(D_t+ Q \big)u(t,x) = f(t,x),
\end{equation*}
with $f \in C^\infty(\T\times M)$, is equivalent to the sequence of ordinary differential equations
\begin{equation}\label{syst-abstrac}
\big( D_t+ Q_j^\top \big) \widehat{U}_j(t) = \widehat{F}_j(t),
\end{equation}
where $\widehat{U}_j(t) = \big(\widehat{u}_j^1(t), \ldots, \widehat{u}_j^{\, d_j}(t)\big)$ and $\widehat{F}_j(t) = \big(\widehat{f}_j\,\!^1(t), \ldots, \widehat{f}_j\,\!^{d_j}(t)\big)$ with $\widehat{u}_j^{\,k}(t)$ and $\widehat{f}_j\,\!^{k}(t)$ in $C^\infty (\T)$, for $1 \leqslant k \leqslant d_j$ and $j \in \N_0$.

In order to study the behavior of the solutions of equation $Lu =f$ in function of the sequences of eigenvalues generate by the  matrices $Q_j$, we introduce the following definition.

\begin{definition}\label{strong-diag}
	We say that an $E-$invariant operator $Q$ is strongly diagonalizable, if there exists $j_0 \in \N$ and a sequence of matrices $\{S_j\}_{j \geqslant j_0}$ such that
	\begin{equation*}
	Q_j  = S_j \, \mbox{diag}(\sigma_j^1, \dots, \sigma_j^{d_j})\, S_j^{-1}, \ j \geqslant j_0
	\end{equation*}
	satisfying
	\begin{equation*}
	\|S_j \| \leqslant k j^{\, s }  \  \textrm{ and } \ \|S_j^{-1} \| \leqslant k j^{\, r},
	\end{equation*}
	for some $k >0$ and  $r,s \in \R$, where $\|\cdot\|$ is the usual operator norm.
\end{definition}

\begin{example}\label{example-normal-case}
	Let $Q$ be a normal $E-$invariant operator, that is\/ $QQ^*=Q^*Q$, and write $Q = A+iB$, where
	$$ A = \frac{Q+Q^*}{2} \mbox{ \ and \ } B = \frac{Q-Q^*}{2i}.$$
	It is easy to verify that the operators $A$ and $B$ are self-adjoint and
	$$[A, B] = [A, E] = [B, E] =0.$$
	
	Let us denote by $A_j$ and $B_j$ the matrix representations of the restrictions of the $E-$invariant operators $A$ and $B$ to the eigenspaces $E_{\lambda_j}$, and by $\mu_j^k$ and $\nu_j^k,$ with $1\leqslant k\leqslant d_j$, the (real) eigenvalues of $A_j$ and $B_j$, respectively.
	
	By a known result of linear algebra, see section 6.5 of \cite{Hof}, for each $j \in \N_0$, there is a unitary matrix $S_j$ such that
	$$Q_j  = S_j \, \mbox{diag}(\mu_j^1+i\nu_j^1, \dots, \mu_j^{d_j} +i\nu_j^{d_j})  S_j^{-1}.$$
	
	This shows that any normal $E-$invariant operator defined on $M$  is strongly diagonalizable.
\end{example}

Turning back to \eqref{syst-abstrac}, let us assume that $Q$ is strongly diagonalizable. Then, for each $j \in \N_0$, that equation is equivalent to the diagonal system
\begin{equation}\label{general-syst-v}
\left (
\begin{array}{cccc}
\partial_t + i \sigma_j^1       & 0                          &  \ldots & 0                            \\
0                      & \partial_t + i \sigma_j^2  &  \ldots & \vdots                            \\
\vdots                 & \vdots                     &  \ddots & 0                       \\
0                      & \ldots                          &  0      & \partial_t + i \sigma_j^{d_j}
\end{array}
\right) \widehat{V}_j(t) = \widehat{G}_j(t),
\end{equation}
where $\widehat{V}_j(t) = i S_j \widehat{U}_j(t)$, $\widehat{G}_j(t) = i S_j \widehat{F}_j(t)$ and $t \in \T$.

Moreover, there are constants $C,r$ and $s$, and $j_0 \in \N$, such that, if $j \geqslant j_0$ then
\begin{equation}\label{v=u}
\|\widehat{V}_j(t)\| \leqslant C j^{s} \, \|\widehat{U}_j(t)\| \ \textrm{ and } \ \|\widehat{U}_j(t)\| \leqslant C j^{r} \, \|\widehat{V}_j(t)\|, \ t \in \T.
\end{equation}

Therefore the entries of the vector $\widehat{V}_j(t)$ have the same type of growth (or decay) as the entries of the vector $\widehat{U}_j(t)$, when $j \to \infty$. Clearly, the same relation holds between to the entries of the vectors $\widehat{G}_j(t)$ and $\widehat{F}_j(t)$.

For the sake of convenience of notation, let us reorder the terms of sequence $\{\sigma_k^{j}\}$ in the following way
\begin{equation}\label{mu-nu-ell}
\{\sigma_\ell\}_{\ell \in \N_0} \doteq  \ \{ \sigma_1^1, \ldots, \sigma_1^{d_1}, \sigma_2^{1}, \ldots, \sigma_2^{d_2}, \ldots, \sigma_j^{1}, \ldots, \sigma_j^{d_j}, \ldots \}.
\end{equation}

\begin{proposition}\label{Gamma-finite}
	If $L$ is (GH), then the set\, $\Gamma_Q = \{\ell \in \N; \ \sigma_\ell \in \Z\}$ is finite.
\end{proposition}
\begin{proof}
	If we had $\Gamma_Q$ infinite, then it would be possible to obtain an infinite subset $\{ {\ell}_j \in \N; \sigma_{\ell_{j}} \in \Z\}$ and construct a sequence of  functions $v_\ell \in C^{\infty}(\T)$ defined by $v_\ell(t) = \exp(-i\sigma_{\ell} t),$ if $\ell = \ell_j,$ for some $j \in \N$, and $v_\ell(t) \equiv 0$ otherwise.
	
	Since $|v_{\ell_j}(t)|=1,$ for all $j \in \N$, by Proposition \ref{prop-smooth-2} we would have
	$$v = \sum_{\ell \in \N_0} v_\ell(t) e_{\ell}(x) = \sum_{j \in \N_0} \exp(-i \sigma_{\ell_j} t) e_{\ell_j}(x) \in \D(\T \times M) \setminus C^{\infty}(\T \times M).$$
	
	Finally, from \eqref{general-syst-v}, we would have $L v  = 0$ implying that $L$ is not (GH).
	
\end{proof}

In view of the last Proposition, in the study of the global hypoellipticity of $L$, we may assume without loss of generality that $\sigma_\ell \notin \Z$, for all $\ell \in \Z$.

In this situation, the equation
\begin{equation}\label{eq-el1}
(\partial_t + i \sigma_\ell)v_\ell(t) =  g_\ell(t)
\end{equation}
has a unique solution given by
\begin{equation}\label{sol-1}
v_\ell(t) = \dfrac{1}{1 - \exp(- 2 \pi i \sigma_\ell)} \int_{0}^{2 \pi} \exp(- i \sigma_\ell s) g_\ell(t -s) ds,
\end{equation}
or equivalently by
\begin{equation}\label{sol-2}
v_\ell(t) = \dfrac{1}{\exp(2 \pi i \sigma_\ell) -1 } \int_{0}^{2 \pi} \exp( i \sigma_\ell s) g_\ell(t + s) ds,
\end{equation}
for each $\ell \in \N_0$.

\begin{lemma}\label{lt2}
	There are positive constants $C,$ $M,$ and $R$ such that
	\begin{equation*}\label{2222-a}
	|1 - \exp(- 2 \pi i \sigma_{\ell})| \geqslant C \ell^{-M}, \ \ell \geqslant R,
	\end{equation*}
	if there exist positive constants $C', M',$ and $R'$ such that
	\begin{equation*}\label{2222}
	\inf_{\tau \in \Z}  |\tau + \sigma_{\ell}|  \geqslant C' \ell^{-M'}, \ \ell \geqslant R'.
	\end{equation*}
\end{lemma}

The proof of this result follows the same ideas of the proof given in Proposition 5.7 of \cite{AvGrKi}.

\begin{theorem}\label{general-GH}
	Let $Q$ be strongly diagonalizable operator with matrix symbol of moderate growth and $\Gamma_Q$  finite. Then $L = D_t + Q$ is (GH) if and only if there are constants $C,\theta>0$ and $\ell_0 \in \N_0$ such that
	\begin{equation}\label{nonLiouville-GW}
	\inf_{\tau \in \Z} |\tau + \sigma_\ell| \geqslant C \ell^{-\theta}, \   \ell \geqslant \ell_0.
	\end{equation}
	
	\noindent In other words, $L$ is (GH) if and only if there are $C, \theta>0$ and $\ell_0 \in \N_0$ such that either
	\begin{equation}\label{nonLiouville}
	\inf_{\tau \in \Z} |\tau + \Rea(\sigma_{\ell})| \geqslant C \ell^{-\theta},  \mbox{ or  } \ |\Ima(\sigma_{\ell})|  \geqslant C \ell^{-\theta},
	\end{equation}
	for all $\ell \geqslant \ell_0.$
\end{theorem}

\begin{proof} Firstly, assume that \eqref{nonLiouville-GW} holds true and let us prove that $L$ is (GH). Since $Q$ is strongly diagonalizable, by \eqref{general-syst-v} and \eqref{v=u}, it is enough to show that: if $v \in \D(\T\times M)$ and $Lv \in C^\infty(\T\times M)$,  then  $v \in C^\infty(\T\times M)$.
	
	We may assume $\sigma_\ell \notin \Z$, for all $\ell \in \N$, therefore the solution of the equations \eqref{eq-el1} are given by the
	equivalent expressions  \eqref{sol-1} or \eqref{sol-2}. The choice of each one of this expressions depends on the sign of $\Ima(\sigma_\ell)$.
	
	More precisely, if $\Ima(\sigma_\ell)<0$ then by \eqref{2-prop-smooth-2} and \eqref{sol-1} we have
	\begin{eqnarray*}
		|\partial^{\alpha}_t v_\ell(t)| &\leqslant& |1 - \exp(- 2 \pi i \sigma_\ell)|^{-1} \int_{0}^{2 \pi} \exp( \Ima(\sigma_\ell) s) \left|\partial^{\alpha}_t g_\ell(t -s) \right|ds \\
		&\leqslant& |1 - \exp(- 2 \pi i \sigma_\ell)|^{-1} \ 2\pi \ \max_{t\in \T}\|\partial_t^{\alpha} g_\ell(t)\|_{\C^{d_j}}.
	\end{eqnarray*}
	
	By Lemma \ref{lt2}, there are positive constants $C,$ $M,$ and $R$ such that $$|1 - \exp(- 2 \pi i \sigma_{\ell})|^{-1} \leqslant C \ell^{M}, \ \ell \geqslant R,$$
	
	Since $g$ is a smooth function, given $N\in \N$ and $\alpha \in \N_0$ we obtain  positive constants $C'$ and $R'>0$ such that
	\begin{equation*}
	\max_{t\in \T}\|\partial_t^{\alpha} g_\ell(t)\|_{\C^{d_j}} \leqslant C' \ell^{-(N+M)},  \mbox{ for all }  \ell \geqslant R'.
	\end{equation*}
	
	Therefore $|\partial^{\alpha}_t v_{\ell}(t)| \leqslant C'' \ell^{-N}$, when $\ell  \geqslant R''$ and  $\Ima(\sigma_\ell)<0$.
	Analogously, when $\Ima(\sigma_\ell)>0$,  we use \eqref{sol-2} and obtain the same type of estimate.  Thus $\{v_\ell(t)\}$ satisfy the estimate \eqref{2-prop-smooth-2}, for any $\alpha \in \N_0$ and natural $N$. It follows that $v \in C^\infty (\T\times M)$ and $L$ is (GH).

	Conversely, assume that the estimate in \eqref{nonLiouville-GW} fails. Then, there is a subsequence $\{\sigma_{{\ell}_n}\}$ and a sequence $\{\tau_n\} \subset \Z$ such that $$ |\sigma_{{\ell}_n} - \tau_n | < {{\ell}_n}^{-n/2}, \ n \to \infty.$$

	Since $Q$ has moderated growth, there is $\beta>0$ such that $|\sigma_{\ell}| = \mathcal{O}(\ell^{\beta})$. It follows from the last estimate that
	\begin{equation}\label{non-Dioph-2}
	|\tau_n | = \mathcal{O}({{\ell}_n}^{-n/2 + \beta}),  \  n \to \infty.
	\end{equation}

	Next, define sequences of functions $\{v_\ell(t)\}$ and $\{g_\ell(t)\}$, with ${\ell \in \N_0}$, by
	\begin{eqnarray*}
		v_\ell(t)  =  \left\{
		\begin{array}{l}
			e^{-i\tau_n t}, \textrm{ if } \ell = \ell_n, \\[2mm]
			0, \textrm{ otherwise. }
		\end{array} \right.
		& \mbox{ and } &
		g_\ell(t)  =  \left\{
		\begin{array}{l}
			(\sigma_{{\ell}_n} - \tau_n)e^{-i\tau_n t}, \textrm{ if } \ell = \ell_n, \\[2mm]
			0, \textrm{ otherwise. }
		\end{array} \right.
	\end{eqnarray*}
	
	Note that $|v_{\ell_n}(t)| \equiv 1$, and by (\ref{non-Dioph-2}), $|\partial^{\alpha}_t v_{\ell_n}(t)| =  |\tau_n|^{\alpha} \leqslant C  {\ell_n}^{-n/2+ \beta},$  when $n \to \infty,$ therefore
	$$ v = \sum_{\ell \in \N} v_\ell(t)e_{\ell}(x) =  \sum_{n \in \N}e^{-i\tau_n t}e_{\ell_n}(x) \in \D(\T \times M) \setminus C^{\infty}(\T \times M).$$
	
	However, for any  $\alpha \in \N_0$, we have
	$$
	|\partial^{\alpha}_t g_{\ell_n}(t)| \leqslant  |\tau_n|^{\alpha} |\sigma_{{\ell}_n} - \tau_n| \leqslant C  {{\ell}_n}^{-n/2}
	{{\ell}_n}^{-n/2 + \beta} \leqslant C  {{\ell}_n}^{-n + \beta},
	$$
	when $n \to \infty,$ therefore
	\begin{equation*}
	g(t) = \sum_{\ell \in \N} g_\ell(t)e_{\ell}(x) =  \sum_{n \in \N}(\sigma_{{\ell}_n} - \tau_n)e^{-i\tau_n t}e_{\ell_n}(x) \in C^{\infty}(\T \times M)
	\end{equation*}
	and $L$ is not (GH).
	
\end{proof}
\begin{remark}
	From the last theorem, there are only two types of (GH) operators of the form $L=D_t + Q $, namely:
	\begin{itemize}\label{typeI+II}
		\item[] {Type\! I.} there are $C,\theta>0$ and $\ell_0 \in \N_0$ such that
		$$\displaystyle\inf_{\tau \in \Z} |\tau + Re(\sigma_{\ell})| \geqslant C \ell^{-\theta}, \ \ell \geqslant \ell_0; $$
		
		\item[] {Type\! II.} there are $C,\theta>0$ and $\ell_0 \in \N_0$ such that
		$$ |Im(\sigma_{\ell})|  \geqslant C \ell^{-\theta},  \ \ell \geqslant \ell_0. $$
	\end{itemize}
	
	This characterization is in line with results of J. Hounie (see \cite{Hou79} Section 2) and A. Bergamasco (see \cite{BERG94} Section 3).
	
\end{remark}

\section{Perturbation of vector fields by low order terms \label{section4}}

There are three seminal articles of S. Greenfield and N. Wallach (see \cite{GW1},\cite{GW3}, and \cite{GW2}) that aroused the interest of geometers and analysts by the study of global hypoellipticity on closed manifolds, establishing a relation between this property and the behavior of the spectrum at infinity. In the particular case of the torus, it was observed that an obstruction of number-theoretical nature appears as a necessary condition for the global hypoellipticity

More precisely, consider the vector field
\begin{equation*}
L = D_t + \omega D_x, \ \omega=\alpha+i\beta \in \C,
\end{equation*}
where $(t, x) \in \T^2$ and $D_x = - i \partial/\partial x$. In \cite{GW1}, the authors showed that $L$ is (GH) on $\T^2$  if and only if either $\beta\neq 0$ or $\alpha$ is an irrational non-Liouville number.

We recall that $\alpha$ is a Liouville number if it can be approximated by rationals to any order, that is, for every positive integer $N$, there is $C > 0$, and infinitely many integer pairs $(p, q)$ so that: $|\alpha-p/q|<C/q^{N}.$

In this session, we show how to recover this result from the viewpoint of $E-$invariant operators and how Theorem \ref{general-GH} can be used to study perturbations by lower order operators. In particular, we recovered results of Bergamasco of \cite{BERG94} and extended them to a broader class of perturbations.

\subsection{Constant Vector fields and zero-order perturbations on $\T^2$} \ \medskip 

We start by observing that $Q=\omega D_x$ commutes with $E = -D_x^2$ on $\T_x^1$. Since the spectrum is  $spec(E) = \{j^2; j \in \N_0 \}$; the eigenfunctions may be chosen as $ e_j^1(x)=e^{-ijx}$ and  $e_j^2(x)=e^{ijx}$; and the eigenspaces are $E_0=\C$ and $E_{j^2} = span \{e^{-ijx}, e^{ijx}\}$, for $j\in \N$. Obviously $d_0=1$ and $d_j = 2$ for all $ j \geqslant 1$.

For the sake of simplicity, we will use any of the three expressions below for the Fourier expansion of a distribution $u \in \D(\T)$
$$
u \ = \ \widehat{u}_0^{\, 1} + \sum_{j \in \N} \sum_{k=1}^{2} \widehat{u}_j^{\, k} e_j^k(x) \ = \ \sum_{j \in \N_0} \sum_{k=1}^{2} \widehat{u}_j^{\, k} e_j^{\, k}(x) \ = \ \sum_{j\in \N_0} \left\langle \widehat{u}_{j},\overline{ e_{j}}(x)  \right\rangle_{\C^2},
$$
where we set $\widehat{u}_0^{\, 2}  \doteq 0$, and denote $\widehat{u}_{j} = \left ( \widehat{u}_{j}^{\, 1}, \widehat{u}_{j}^{\, 2} \right)$, and
$e_{j}(x) = \left ( e_{j}^{\, 1}(x), e_{j}^{\, 2}(x) \right)$, for any $j \in \N_0$.

In particular, the restriction of $D_x$ to each eigenspace $E_{j^2},$ $j\in\N,$ can be represented by ${\mathcal{D}}_{j} = \mbox{diag}(-{j} \, , \, {j})$; thus the equation $(D_t+\omega D_x) u = f$ is equivalent to the system
\begin{equation*}
D_t \widehat{U}_{j}(t) + \omega {\mathcal{D}}_{j}\widehat{U}_{j}(t)  = \widehat{F}_{j}(t), \ {j} \in \N,
\end{equation*}
where $t\in \T$, and $\omega=\alpha+i\beta \in \C.$

It follows by Theorem \ref{general-GH} that $L=D_t+\omega D_x$ is globally hypoelliptic if, and only if, there are constants $C,\theta>0$ and $j_0 \in \N_0$ such that
\begin{equation*}
\inf_{\tau \in \Z} |(\tau + j \alpha) +ij\beta| \geqslant C j^{-\theta}, \mbox{ if }  j \geqslant j_0,
\end{equation*}
which is equivalent to say that either $\beta\neq 0$ or $\alpha$ is an irrational non Liouville number.

\begin{remark}\label{cte-pert}
	A more careful reading of this result allows us to extract information about the global hypoellipticity of the perturbed operator  $L_{\epsilon} = D_t + \omega D_x +  \epsilon,$ with $\epsilon  \in \C.$ Indeed,
	\begin{enumerate}
		\item[$i.$] when $\beta \neq 0$, the eigenvalues $\sigma_{j}^k$ of the matrices $\omega\mathcal{D}_j + \epsilon$ satisfy
		$\Ima(\sigma_j^1) = \Ima(\epsilon - j\omega) \ \textrm{ and } \ \Ima(\sigma_j^2) = \Ima(\epsilon + j\omega).$ Therefore, $L_{\epsilon}$ is (GH), for all $\epsilon \in \C$. 
		
		\item[$ii.$] when $\beta=0$, $L_{\epsilon}$ is (GH) if and only if there is $\theta >0$ and $\ell_0 \in \N$ such that
		\begin{equation*}
		\inf_{\tau  \in \Z}  \big|\tau  \pm \alpha \ell  + \epsilon \big| \geqslant C\ell^{-\theta}, \ \ell \geqslant \ell_0.
		\end{equation*}	
	\end{enumerate}
\end{remark}

Finally, by using Propositions 3.1 and 3.2 of \cite{BERG94}, we obtain the following proposition:

\begin{proposition}
	Given $\omega=\alpha+i\beta\in \C$ and $\epsilon=\epsilon(t,x)\in C^\infty(\T^2)$, consider the operator
	$$
	L_{\epsilon} = D_t + \omega D_x +  \epsilon(t,x).
	$$
	Setting ${L}_{\epsilon_0} = D_t + \omega D_x +  \epsilon_0$,  where
	$$
	\epsilon_0 = \frac{1}{(2\pi)^{2}}\int_{0}^{2\pi}\!\! \int_{0}^{2\pi} \epsilon(t,x) dtdx,
	$$
	we have
	\begin{enumerate}
		\item[$i.$] if $\beta \neq 0$ then $L_{\epsilon}$ is (GH);
		\item[$ii.$] if  $\beta=0$ and $\alpha$ is an irrational non Liouville number then the following statements are equivalent:
		\begin{enumerate}
			\item[a.] $L_{\epsilon}$ is (GH);
			\item[b.] ${L}_{\epsilon_0}$ is (GH);
			\item[c.] there is $\theta >0$ and $\ell_0 \in \N$ such that
			\begin{equation}
			\inf \Big\{ \big|\tau  \pm \alpha \ell  + \epsilon_0 \big|; \tau  \in \Z \Big\} \geqslant C\ell^{-\theta}, \forall \ell \geqslant \ell_0.
			\end{equation}
		\end{enumerate}
	\end{enumerate}
\end{proposition}

In the next subsection we study perturbations of  $L= D_t + \omega D_x$ by $E-$invariant operators.

\subsection{Perturbations invariant with respect to $E$}  \ \medskip

Let us start by considering a class of $E-$invariant operators $\D(\T)$ through  its matrix symbol.

Let $\{\mathcal{R}_j; j \in \N_0\}$ be a sequence of matrices given by
\begin{equation*}\label{seq-self-adj-matrix}
{\mathcal{R}}_0 \in \C,
\ \textrm{ and } \
{\mathcal{R}}_j = \left [
\begin{array}{cc}
r_{j}^{1}   & \tau_j^{2} \\[2mm]
\tau_j^{1}  & r_{j}^{2}
\end{array} \right] \in \C^{2\times 2},  \ j \in \N,
\end{equation*}
satisfying  the moderate growth condition
\begin{equation}\label{delta-increase}
\|{\mathcal{R}}_j\| = \mathcal{O}(j^{\, \delta}), \mbox{ as } j \to \infty,
\end{equation}
for some $\delta \in \mathbb{R}$.

For each $u \in \D(\T)$ we define
\begin{equation}\label{op-A}
{\mathcal{R}} u \doteq \sum_{j\in \N_0} \left\langle \widehat{u}_{j}, {\mathcal{R}}_j e_j(x)  \right\rangle_{\C^2}.
\end{equation}

\begin{proposition}\label{R-cont}
	The operator ${\mathcal{R}}: \H^{s}(\T) \to \H^{s - \delta}(\T)$, defined in \eqref{op-A}, is linear and continuous, for all $s \in \Z$.
\end{proposition}
\begin{proof} We start by writing $\mathcal{R} = \mathcal{R}_1 + \mathcal{R}_2$ where
	$$\mathcal{R}_1 u \doteq \sum_{j\in \N}  \sum_{k=1}^2 \widehat{u}_j^{k}  r_j^k  e_{j}^{k}(x) \ \mbox{ and } \ \mathcal{R}_2  u \doteq \sum_{j\in \N}\sum_{k=1}^2 \widehat{u}^{\, 3-k}_j \tau_j^k  e_{j}^{k}(x),$$
	and then we prove that $\mathcal{R}_1$  and $\mathcal{R}_2$ are continuous from $\H^{s}(\T)$ to $\H^{s - \delta}(\T)$.
	
	If $u \in \H^s(\T)$ then by \eqref{Sobolev1matr} and \eqref{delta-increase} we have
	\begin{eqnarray*}
		\|\mathcal{R}_1 u\|^2_{\H^{s - \delta}} & = & \sum_{j \in \N} \sum_{k=1}^{2} |\widehat{u}_j^k r_j^k |^2  j^{2(s - \delta)} \\
		& \leqslant & C \sum_{j \in \N} \sum_{k=1}^{2}  |\widehat{u}_j^k|^2 j^{2(s - \delta)+2\delta}\! =\! C\|u\|^2_{\H^{s}}
	\end{eqnarray*}
	and
	\begin{eqnarray*}
		\|\mathcal{R}_2 u\|^2_{\H^{s - \delta}} & = & \sum_{j \in \N} \sum_{k=1}^{2} | \widehat{u}^{\, 3-k}_j \tau_j^k|^2  j^{2(s - \delta)}\\
		& \leqslant & C \sum_{j \in \N} \sum_{k=1}^{2}  |\widehat{u}_j^{\, 3-k}|^2 j^{2(s - \delta)+2\delta}\!\!=C\|u\|^2_{\H^{s}}.
	\end{eqnarray*}
	
	This concludes the proof of the continuity.
	
\end{proof}

For the purposes of this section, it is sufficient to consider the case where
\begin{equation*}\label{matrix-example}
{\mathcal{R}}_0 = 0,  \
{\mathcal{R}}_j = \left [
\begin{array}{cc}
r_{j}   & \gamma_j \\
\gamma_j  & r_{j}
\end{array}
\right] \in \R^{2\times 2}, \ j \in \N,
\end{equation*}
and
\begin{equation*}
\|{\mathcal{R}}_j\| = \mathcal{O}(j^{\delta}), \ j \to \infty, \textrm{ for some } 0 \leqslant \delta < 1.
\end{equation*}

\smallskip
If $u \in \D(\T^2)$ is a solution of equation  $(D_t+\omega D_x + \mathcal{R}) u = f$, with $f \in C^\infty(\T^2)$,  then the $x-$Fourier coefficients of $u$ satisfy the system
\begin{equation*}\label{sys-pert-1}
D_t \widehat{U}_j(t) + ( \omega {\mathcal{D}}_j + {\mathcal{R}}_j)\widehat{U}_j(t)  = \widehat{F}_j(t), \ j \in \N.
\end{equation*}

Observe that the operator $(\omega D_x + \mathcal{R})$ is strongly diagonalizable, since each matrix $(\omega{\mathcal{D}}_j + {\mathcal{R}}_j)$ is symmetric, for $j\in\N$. In particular,
\begin{equation*}\label{characteristic-polynomial}
p_j(\xi) = \xi^2 - 2r_j\xi + (r_j^2 - \omega^2j^2) - \gamma_j^2.
\end{equation*}
is the characteristic  polynomial of $\omega {\mathcal{D}}_j + {\mathcal{R}}_j$, for each $j.$

We say that a perturbation $\mathcal{R}$ is commutative when
$$[D_x, \mathcal{R}]=0,$$
otherwise, we say that $\mathcal{R}$ is a commutative perturbation.

The examples we are considering satisfy
\begin{equation*}\label{bracket-R}
[{\mathcal{D}}_j , {\mathcal{R}}_j] = 2 j \gamma_j
\left [
\begin{array}{cr}
0   & -1   \\
1   &  0
\end{array}
\right],
\end{equation*}
thus
\begin{equation*}\label{commutative-case}
[D_x , \mathcal{R}] = 0  \  \Leftrightarrow  \  \gamma_j =0, \ \forall j \in \N_0. 	
\end{equation*}

Therefore, if $\mathcal{R}$ is a commutative perturbation, we have
\begin{equation*}\label{eigen=sum}
\omega {\mathcal{D}}_j + {\mathcal{R}}_j =
\left [
\begin{array}{cr}
-  \omega j + r_{j}   & 0 \\
0                      &    \omega j + r_{j}
\end{array}
\right], \ \forall j \in \N,
\end{equation*}
and it is possible to use  Theorem \ref{general-GH} directly to study  the global hypoellipticity of $(D_t+\omega D_x + \mathcal{R})$, as has been done in Remark \ref{cte-pert}.

However, in the case where $\mathcal{R}$ is a non-commutative perturbation, the situation requires particular attention, since we have to ensure that the growth of the sequence $\{S_j\}$ that diagonalizes $(\omega {\mathcal{D}}_j + \mathcal{R}_j)$ is at most polynomial.

The next example presents a non-trivial situation in which this control is done with details.

\begin{example}\label{n-comm-example1}
	Consider the operator $D_t+ \omega D_x + \mathcal{R}$ on $\T^2$, with $\omega\doteq \alpha + i\beta \in  \C$, and $\mathcal{R}$ defined by the sequence of matrices
	\begin{equation*}
	{\mathcal{R}}_0 = 0, \ \textrm{ and } \
	{\mathcal{R}}_j = \left [
	\begin{array}{cc}
	0   & \gamma_j \\
	\gamma_j  & 0
	\end{array}
	\right], \ j \in \N,
	\end{equation*}
	where $\gamma_j \in \R$ and $|\gamma_j| \sim j^{\delta}$, as $j \to \infty$, for some $0 \leqslant \delta < 1$.
	
	When $\beta=0$ and $\alpha\neq 0$, the eigenvalues are given by $\xi_j  = \pm \, (j^2 \alpha^2 + \gamma_j^2)^{1/2}$. In this case, $\mathcal{R}_j$ is symmetric and the matrices $\alpha {\mathcal{D}}_j + {\mathcal{R}}_j$ are diagonalizable by unitary matrices $S_j$, thus
	\begin{equation*}
	\alpha \mathcal{D}_j + \mathcal{R}_j = S_j
	\left (
	\begin{array}{cr}
	\sqrt{j^2 \alpha^2 + \gamma_j^2}     &  0  \\
	0          &   - \, \sqrt{j^2 \alpha^2 + \gamma_j^2}
	\end{array}
	\right)
	S_j^{-1}.
	\end{equation*}

	When $\alpha =0$ and $\omega = i\beta \neq 0$, the eigenvalues are given by $\xi_j  = \pm \, (\gamma_j^2 - \beta^2j^2 )^{1/2}$, and
	\begin{equation}\label{eig-val-1}
	\gamma_j^2 -\beta^2 j^2 = j^2 \left( \dfrac{\gamma_j^2}{{j^2}}  -\beta^2 \right) < 0, \ \textrm{ as } \ j \to \infty,
	\end{equation}
	since $|\gamma_j| \leqslant Cj^{\delta}$, with $\delta<1$.

	Thus, there exists $j_0 \in \N$ such that the eigenvalues are given by
	\begin{equation}\label{eigenv-non=comm-beta-nao-zero}
	\xi_j = \pm i \, (\beta^2 j^2 - \gamma_j^2)^{1/2}, \ j \geqslant j_0,
	\end{equation}
	and the matrices are diagonalizable, that is,
	\begin{equation*}
	i\beta \mathcal{D}_j + \mathcal{R}_j = S_j
	\left (
	\begin{array}{cr}
	i \, (j^2 \beta^2 - \gamma_j^2)^{1/2}     &  0  \\[3mm]
	0          &   -i \, (j^2 \beta^2 - \gamma_j^2)^{1/2}
	\end{array}
	\right)
	S_j^{-1},
	\end{equation*}
	for $j\geqslant j_0$, with
	\begin{equation*}
	S_j =
	\left [
	\begin{array}{cc}
	-i   &  i  \\
	\dfrac{j \beta + (j^2\beta^2 -\gamma_j^2)^{1/2}} {\gamma_j}   & \dfrac{- j \beta + (j^2\beta^2 -\gamma_j^2)^{1/2}} {\gamma_j}
	\end{array}
	\right].
	\end{equation*}
	
	Observe that, for $j\geqslant j_0$ we have
	\begin{equation*}\label{Sj-1}
	\dfrac{| \pm j \beta + (j^2 -\gamma_j^2)^{1/2}|}{|\gamma_j|} \leqslant \dfrac{j |\beta| + (j^2 \beta^2 - \gamma_j^2)^{1/2} }{|\gamma_j|}  \leqslant 2 |\beta| \dfrac{j}{|\gamma_j|} \leqslant C j^{1 - \delta}.
	\end{equation*}

	Thus $\{\|S_j\|\}$ has at most polynomial growth. For the sequence $\{\|S_j^{-1}\|\}$ we observe that
	\begin{equation*}
	\|S_j^{-1}\| = |det (S_j)|^{-1} \|S_j\| =  2^{-1/2}  \, \left(\beta^2  \frac{j^2}{\gamma_j^2} - 1\right)^{-1/2}\|S_j\|.
	\end{equation*}
	
	Now, given $s = 1 - \delta$, we can increase $j_0 \in \N$, if necessary, in order to obtain $|\gamma_j| < |\beta| j^{1-s}$, thus
	\begin{equation*}
	\left(\beta^2 \dfrac{j^2}{\gamma_j^2} - 1\right)^{-1/2} \!\! \leqslant (j^{2s} - 1)^{-1/2} \leqslant 2 j^{-s}, \ j \geqslant j_0.
	\end{equation*}
	and $\|S_j^{-1}\|  \leqslant C j^{1 - \delta} 2j^{-s} = 2C, \ j \to \infty$.

	To finish this example, from Theorem \ref{general-GH} we have $D_t +  i \beta D_x + \mathcal{R}$ is globally hypoelliptic, and ${\mathcal{L}} = D_t +  \alpha D_x + R$ is (GH) if, and only if, there are positive constants $C$, $\theta$ and $R$ such that
	\begin{equation}\label{real-aprox-non-liouv}
	\inf_{\tau \in \Z} \left| \tau + \sqrt{ j^2 \alpha^2 + \gamma_j^2} \right| \geqslant Cj^{-\theta}, \ j \geqslant R.
	\end{equation}
	
\end{example}

\subsection{Perturbations that destroy the global hypoellipticity} \ \medskip

	As mentioned before,  when $\beta \neq 0$, the global hypoellipticity of  $D_t + (\alpha+i\beta) D_x$ is immune to the perturbations by low order terms.
	However, when $\beta =0$ and $\alpha$ is an irrational non-Liouville number, the statement $ii.$ in Remark \ref{cte-pert} and the expression \eqref{real-aprox-non-liouv} open possibilities to construct perturbations that destroys the global hypoellipticity of $D_t + (\alpha+i\beta) D_x$.

\begin{theorem}\label{existence-nGH}
	For any irrational $\alpha$ there exist a commutative perturbation $\mathcal{S}$ and a non commutative perturbation $\mathcal{R}$ such that
	\begin{equation*}
	{\mathcal{L}}_{\mathcal{S}} = D_t +  \alpha D_x + \mathcal{S} \quad and \quad {\mathcal{L}}_{\mathcal{R}} = D_t +  \alpha D_x + \mathcal{R}
	\end{equation*}	
	are not (GH).
\end{theorem}

\begin{proof}
	Let us start by analyzing the non commutative case. Firstly, assume that $\alpha >0$ and consider two sequences of natural numbers $\{p_k\}$ and $\{q_k\}$ such that
	$$\lim_{k \to \infty} p_k / q_k = \alpha, \mbox{ and } \  p_k / q_k > \alpha,  \forall k \in \N.$$
	
	This can be done, for example, by considering a convenient subsequence of the convergents of the continued fractions of $\alpha$ (see Chapter I of \cite{schmidt}).

	Let us define a sequence $\{\gamma_j\}_{j \in \N}$ by $\gamma_{q_k} = \sqrt{p_k^2 - \alpha^2 \, q_k^2}$ and $\gamma_j = \sqrt{j \ }$, if $j \neq q_k$, $\forall k \in \N$.

	Since
	\begin{align*}\label{to-zero}
	\lim_{k \rightarrow \infty} \dfrac{\gamma_{q_k}}{q_k} = \lim_{k \rightarrow \infty} \sqrt{{p_k^2}/{q_k^2} - \alpha^2 } =0,
	\end{align*}
	then $|\gamma_j| \leqslant C j^{\, \delta}$, with $\delta = 1/2$.
	
	The eigenvalues of $\alpha {\mathcal{D}}_j + {\mathcal{R}}_j$ are given by
	$$\sigma_j  \doteq \sqrt{j^2\alpha^2 + \gamma_j^2}, \ j \in \N,$$
	therefore
	\begin{equation*}
	\sigma_{q_k}  = \sqrt{q_k^2 \alpha^2 + \gamma_{q_k}^2} = \sqrt{\alpha^2 q_k^2  + (p_k^2 - \alpha^2 \, q_k^2)} = p_k \in \N.
	\end{equation*}
	
	It follows from Theorem \ref{general-GH} that  ${\mathcal{L}}_{\mathcal{R}}$ is not (GH).
	
	\smallskip
	
	Now, if $\alpha<0$, simply choose  sequences $p_k \in \Z$ and $q_k \in \N$ such that $p_k / q_k \to \alpha$, and $p_k / q_k < \alpha,$ for each $k \in \N.$
	
	As before, we define the sequence $\{\gamma_j\}_{j \in \N}$ by $\gamma_{q_k} = \sqrt{\alpha^2 \, q_k^2 - p_k^2}$ and $\gamma_j = \sqrt{j \ }$, if $j \neq q_k$, $\forall k \in \N$, and repeat the  same procedure above to show that ${\mathcal{L}}_{\mathcal{R}}$ is not (GH).
	
	In the commutative case we repeat the same arguments to the sequence $\{r_j\}_{j \in \N}$. For example, when $\alpha>0$ we set $r_{q_k} \doteq \alpha q_k - p_k,$  and $r_j \doteq \sqrt{j \ }$, if $j \neq q_k$, $\forall k \in \N$. It follows from Theorem \ref{general-GH} that ${\mathcal{L}}_{\mathcal{S}}$ is not (GH).
	
\end{proof}

\section{Perturbations of normal operators \label{section5}}

In the previous section we presented examples of perturbations whose analysis depended essentially on the discovery of the eigenvalues of the sum of two operators.
However, the problem of finding the eigenvalues of a sum of operators is a highly nontrivial problem that does not have a complete answer, even in simpler cases such as the sum of Hermitian or normal matrices.  For more details about this subject, we refer the reader to \cite{Day-98,Klya-98,Tomp-76,Tomp-71,Helm-55}, where the authors investigate the spectrum of sums of matrices.

In this section the idea is to approach this problem from a different point of view, namely, let us assume that the eigenvalues and eigenvectors of the perturbed operator $Q(\epsilon)$ depend analytically on the eigenvalues and eigenvectors of the original operator $Q$.

Let $Q$ and $\mathcal{R}$ be $E-$invariant operators defined on the closed manifold $M$, and consider the perturbed operators
\begin{equation}\label{Q_epsilon}
Q(\epsilon) = Q + \epsilon \mathcal{R},
\end{equation}
for $\epsilon \in \C,$ and
\begin{equation}\label{L-perturbed}
L(\epsilon) = D_t + Q(\epsilon).
\end{equation}

Thus $Q(0)= Q$ and $L(0) = D_t + Q(0) = L$  stands for the unperturbed operator.
The restrictions of $Q(\epsilon)$ to the eigenspaces of $E$ are given by
\begin{equation*}
Q_{j}(\epsilon) \doteq Q(\epsilon)  \Big|_{E_{\lambda_j}}, \ j \in \N_0.
\end{equation*}

From now on, we assume that $Q(\epsilon)$ is a normal operator, for $\epsilon$  sufficiently small.

Following our usual procedure, if $u \in \D(\T \times M)$ satisfies the equation
$$(D_t + Q(\epsilon)) u(t,x) = f(t,x),$$
with $f \in C^\infty(\T\times M)$, then its $x-$Fourier coefficients are solutions of the system of equations
\begin{equation}\label{syst-abstrac-pert}
\big( D_t+ Q_{j}(\epsilon)^\top \big) \widehat{U}_j(t) = \widehat{F}_j(t),
\end{equation}
where $\widehat{U}_j(t) = \big(\widehat{u}_j^1(t), \ldots, \widehat{u}_j^{\, d_j}(t)\big)$ and $\widehat{F}_j(t) = \big(\widehat{f}_j\,\!^1(t), \ldots, \widehat{f}_j\,\!^{d_j}(t)\big)$ with $\widehat{u}_j^{\,k}(t),\widehat{f}_j\,\!^{k}(t)$ in $C^\infty (\T)$, for $1 \leqslant k \leqslant d_j$ and $j \in \N_0$.

By Example \ref{example-normal-case}, the operator $Q(\epsilon)$ is strongly diagonalizable by a sequence of unitary matrices $\{S_{j}(\epsilon)\}_{j\in \N_0}$, therefore \eqref{syst-abstrac-pert} is equivalent to the diagonal system
\begin{equation*}
D_t \widehat{V}_j(t)+ \mbox{diag} (\sigma_{j}^{1}(\epsilon), \ldots , \sigma_{j}^{d_{j}}(\epsilon)) \widehat{V}_j(t)  = \widehat{G}_j(t).
\end{equation*}
where $\widehat{V}_j(t) = i S_j(\epsilon) \widehat{U}_j(t)$, $\widehat{G}_j(t) = i S_j(\epsilon) \widehat{F}_j(t)$ and $t \in \T$.

The next step is to obtain a relation between the eigenvalues  $\sigma_{j}^{m}(\epsilon)$ of $Q_{j}(\epsilon)$ and the corresponding eigenvalues $\sigma_{j}^{m}$ of the unperturbed matrix $Q_{j}$.

Motivated by the theory presented in T. Kato \cite{kato} and F. Rellich \cite{Rellich}, we assume that the eigenvalues and eigenvectors of $Q_{j}(\epsilon)$ have an analytic expansion in the form
\begin{equation}\label{eigen-expansion}
\sigma_{j}^m(\epsilon) = \sigma_{j}^m + \sum_{k=1}^{\infty} \sigma_{j, k}^m \, \epsilon^k \ \textrm{ and }
\ v_{j}^m(\epsilon) = v_{j}^m + \sum_{k=1}^{\infty} v_{j, k}^m \, \epsilon^k,
\end{equation}
where each $\sigma_{j}^m$ is an eigenvalue of the matrix $Q_{j}$ and  $v_{j}^m$ is its corresponding eigenvector.

Obviously, the calculation of coefficients $\sigma_{j, k}^m$ is the core of the problem. Thankfully, the second chapter of Kato's book \cite{kato} is entirely dedicated to this problem. In the next section, we will calculate these coefficients explicitly in some examples.

Now, as in \eqref{mu-nu-ell} we rearrange the terms of the sequences $\sigma_{j}^m(\epsilon)$ by writing
\begin{equation*}
\{\sigma_{\ell}(\epsilon)\}_{\ell\in\N} \doteq  \{ \sigma_{1}^1(\epsilon), \ldots, \sigma_{1}^{d_1}(\epsilon),\sigma_{2}^1(\epsilon), \ldots, \sigma_{2}^{d_1}(\epsilon), \ldots, \sigma_{\ell}^{1}(\epsilon), \ldots, \sigma_{\ell}^{d_\ell}(\epsilon), \ldots \},
\end{equation*}
thus we have
\begin{equation}\label{eigen-expansion2}
\sigma_{\ell}(\epsilon) = \sigma_{\ell} + \sum_{k \in \N} \sigma_{\ell, k}  \, \epsilon^k,
\end{equation}
where $\{\sigma_{\ell}\}$ is the rearrangement of the eigenvalues of $Q$.

Thus, the following result is a consequence  of Theorem \ref{general-GH}.

\begin{theorem}\label{prop-perturbation}
	Let $Q$ and $\mathcal{R}$ be $E-$invariant operators on the manifold $M$, and set $Q(\epsilon) = Q + \epsilon \mathcal{R},$ with $\epsilon \in \C.$
	Assume that $Q(\epsilon)$ is normal in some interval $|\epsilon|< \epsilon_0$ and that $\Gamma_{Q(\epsilon)} = \{\ell \in \N; \ \sigma_{\ell}(\epsilon) \in \Z \}$ is finite.
	
	Then the  operator $L(\epsilon) = D_t + Q(\epsilon)$ is (GH) if, and only if, there are constants $C,\theta>0$ and $\ell_0 \in \N_0$ such that
	\begin{equation}\label{nonLiouville-GW2}
	\inf_{\tau \in \Z} \left|\tau + \sigma_{\ell} + \sum_{k \in \N} \sigma_{\ell, k} \epsilon^k  \right| \geqslant C \ell^{-\theta}, \quad  \ell \geqslant \ell_0.
	\end{equation}
\end{theorem}

\begin{proof}
	Since the operator $Q(\epsilon)$ is normal, it follows from \cite{Jamison54} that the eigenvalues of $Q(\epsilon)$ can be written in the form \eqref{eigen-expansion2}. Thus the result follows directly from Theorem \ref{general-GH}.
	
	Recall that, by Proposition \ref{Gamma-finite}, the requirement of $\Gamma_{Q(\epsilon)}$ to be finite is a necessary condition to study the global hypoellipticity of $L(\epsilon)$.
	
\end{proof}

\begin{remark}
	If the operators $Q$ and $\mathcal{R}$ are simultaneously strongly diagonalizable, that is, booth are strongly diagonalizable by the same sequence of matrices $\{S_j\}_{j\in\N_0}$, then we can discard the hypothesis of normality in the last theorem.
	
	Indeed, in this case the eigenvalues of $Q_{j}(\epsilon) = Q_j + \epsilon \mathcal{R}_j,$ are given by $\sigma^m_j + \epsilon \rho^m_j$, where  $\rho^m_j$ are the eigenvalues of $\mathcal{R}_j$.
	Thus, $L(\epsilon) = D_t + Q(\epsilon)$ is (GH) if, and only if,
	$$
	\inf_{\tau \in \Z} |\tau + \sigma_{\ell} + \epsilon \rho_{\ell}|  \geqslant C {\ell}^{-\theta}, \mbox{ as } {\ell} \to \infty,
	$$
	where $\{\rho_{\ell}\}$  is the rearrangement of the terms $\rho_j^m$, as in \eqref{mu-nu-ell}.
\end{remark}

\begin{example}\label{comm-example}
	Recalling the definitions of operators of Type I and II given in \ref{typeI+II}, if the operators $Q$ and $\mathcal{R}$ are simultaneously strongly diagonalizable, it follows of the above remark that:
	\begin{enumerate}
		\item[$i.$]  if $\mathcal{R}$ is self-adjoint and $L=D_t+Q$ is of Type II then $L(\epsilon)$ is (GH), since
		$Im(\sigma_{\ell}(\epsilon)) = Im( \sigma_{\ell})$.
		
		\item[$ii.$]  if $i\mathcal{R}$ is self-adjoint and $L$ is of Type I, then $L(\epsilon)$ is (GH), since $
		|\tau + Re(\sigma_{\ell}(\epsilon))| = |\tau + Re(\sigma_{\ell})|,$  for all $\tau \in \Z.$
	\end{enumerate}
\end{example}

\begin{remark}
	Analyzing more carefully  the results of this section, it is clear that the normality hypothesis on $Q(\epsilon)$ was used only in two moments: first, to ensure that $Q(\epsilon)$ is strongly diagonalizable, that is, to diagonalize the system of equations \eqref{syst-abstrac-pert} by a sequence of unitary matrices; second, to ensure that the eigenvalues and eigenvectors of $Q(\epsilon)$ can be written in the form \eqref{eigen-expansion}.
	
	Therefore, it is possible to discard the normality hypothesis, requiring that the operator $Q(\epsilon)$ be strongly diagonalizable and that the eigenvalues and eigenvectors of $Q(\epsilon)$ can be written in the form  \eqref{eigen-expansion}.
\end{remark}

In view of this remark, we have the following result.

\begin{theorem}\label{prop-perturbation2}
	Let $Q$ and $\mathcal{R}$ be $E-$invariant operators on the manifold $M$, and set $Q(\epsilon) = Q + \epsilon \mathcal{R},$ with $\epsilon \in \C.$
	Assume that in some interval $|\epsilon|< \epsilon_0$ the operator $Q(\epsilon)$ is strongly diagonalizable, that $\Gamma_{Q(\epsilon)} = \{\ell \in \N; \ \sigma_{\ell} (\epsilon) \in \Z \}$ is finite, and that the eigenvalues of $(\epsilon)$ can be written in the form  \eqref{eigen-expansion}.
	
	Then the  operator $L(\epsilon) = D_t + Q(\epsilon)$ is (GH) if, and only if, there are constants $C,\theta>0$ and $\ell_0 \in \N_0$ such that
	\begin{equation}\label{nonLiouville-GW3}
	\inf_{\tau \in \Z} \left|\tau + \sigma_{\ell} + \sum_{k \in \N} \sigma_{\ell, k} \epsilon^k  \right| \geqslant C \ell^{-\theta}, \quad  \ell \geqslant \ell_0.
	\end{equation}
\end{theorem}

\begin{remark}
	With the same hypothesis as in Theorem \ref{prop-perturbation} (or Theorem \ref{prop-perturbation2}),
	assume that there are positive constants $C_1$ and $\theta_1$, and natural numbers $\ell_0$ and $N$ such that
	\begin{equation}\label{1}
	\inf_{\tau \in \Z}\left|\tau + \sigma_{\ell} +  \sum_{k=1}^{N} \sigma_{\ell, k} \epsilon^k\right| \geqslant C_1 \ell^{-\theta_1},
	\end{equation}
	for $\ell \geqslant \ell_0$ and  $|\epsilon|<\epsilon_0 <\frac{1}{2}$
	
	If there are $\theta_2 > \theta_1$ and $C_2>0$ such that
	\begin{equation}\label{2}
	\sup_{k \geqslant N+1}|\sigma_{\ell,k}| \leqslant C_2 \ell^{-\theta_2}, \quad \ell \geqslant \ell_0
	\end{equation}
	then, for any  $|\epsilon|<\epsilon_0$, the operator $L(\epsilon) = D_t + Q(\epsilon)$ is (GH).
	
	Indeed, under these assumptions we have
	\begin{align*}
	\left|\tau + \sigma_{\ell} + \sum_{k\in \N} \sigma_{\ell, k} \epsilon^k\right| & \geqslant
	\left|\tau + \sigma_{\ell} +  \sum_{k=1}^{N} \sigma_{\ell, k} \epsilon^k\right| -
	\left| \sum_{k = N+1}^{\infty} \sigma_{\ell, k} \epsilon^k\right| \\
	& \geqslant
	C_1 \ell^{-\theta_1} -
	\sup_{k \geqslant N+1}|\sigma_{\ell,k}| \,
	\sum_{k = N+1}^{\infty} |\epsilon|^k \\
	& \geqslant
	C_1 \ell^{-\theta_1}-C_2 \ell^ {-\theta_2} \dfrac{1}{2^N}\\
	& =
	\ell^{-\theta_1}\left(C_1-C_2\dfrac{1}{2^N} \, \ell^{\,\theta_1 -\theta_2}  \right) \\
	& \geqslant
	\ell^{-\theta_1}\left(C_1-C_2\dfrac{1}{2^N} \, \ell_0^{\,\theta_1 -\theta_2}  \right)
	\end{align*}
	Since $\theta_1 <\theta_2$ we can increase $\ell_0$, if necessary, to obtain 
	$$C_1-\dfrac{C_2}{2^N} \, \ell_0^{\,\theta_1 -\theta_2} >0.$$ 
	
	It follows from Theorem \ref{prop-perturbation} (or Theorem \ref{prop-perturbation2}) that  $L(\epsilon) = D_t + Q(\epsilon)$ is (GH).
	
\end{remark}

\section{Analytic perturbations of vector fields on the torus}

In this section we calculate the coefficients $\sigma_{j, k}^m$ of the expansion \eqref{eigen-expansion} of the eigenvalues of $Q_{j}(\epsilon) = \omega \mathcal{D}_j + \epsilon \mathcal{R}_j$
and analyze the global hypoellipticity of  the operator
\begin{equation*}
L(\epsilon) = D_t + \omega D_x + \epsilon \mathcal{R},
\end{equation*}

To this end, we will rely on the notations and results presented mainly in section \ref{section4}.
Recall that $E_0=\C$ and $E_{j^2} = span \{e^{-ijx}, e^{ijx}\}$, for $j\in \N$, are the eigenspaces of $E = -D_x^2$.

Given  a sequence of matrices $\{\mathcal{R}_j; j \in \N_0\}$ with
\begin{equation*}\label{seq-self-adj-matrix2}
{\mathcal{R}}_0 \in \C,
\textrm{ and } \
\mathcal{R}_{j} =
\left [
\begin{array}{cr}
a_j    & b_j \\
c_j    & d_j
\end{array}
\right]  \in \C^{2\times 2},  \ j \in \N,
\end{equation*}
satisfying  the moderate growth condition
$\|{\mathcal{R}}_j\| = \mathcal{O}(j^{\, \delta}),$ for $0 \leqslant \delta < 1,$ we define
\begin{equation*}\label{op-A2}
{\mathcal{R}} u \doteq \sum_{j\in \N_0} \left\langle \widehat{u}_{j}, {\mathcal{R}}_j e_j(x)  \right\rangle_{\C^2},
\end{equation*}
for all $u \in \D(\T)$.

It follows from Proposition \ref{R-cont}  that the operator $\mathcal{R}: \H^{s}(\T) \to \H^{s - \delta}(\T)$ is linear and continuous, for all $s \in \Z$.

The restrictions of $Q(\epsilon)=\omega D_x + \epsilon \mathcal{R},$ to the eigenspaces of $E$ are given by
\begin{equation*}
Q_{j}(\epsilon) \doteq Q(\epsilon)\Big|_{E_{j^2}}, \ j \in \N_0,
\end{equation*}
where $Q(0) = \omega D_x.$

Let us denote the eigenvalues of the matrix $Q_{j}(0)=\omega\mathcal{D}_j$, for $j\in\N$, by
\begin{equation*}\label{wj-simple-eigenv}
\sigma_{j}^1(0) = -\omega j \ \textrm{ and } \ \sigma_{j}^2(0) = \omega j,
\end{equation*}
and the corresponding eigenvectors $v_{j}^1(0) = (1,0)$ and $v_{j}^2(0) = (0,1)$.

The next results show that the eigenvalues and the corresponding eigenvectors of the matrices $Q_{j}(\epsilon)$, for $\epsilon$ small enough, can be written in the form
\begin{equation}\label{exp-eigenv-2}
\sigma_{j}^m(\epsilon) = \sum_{k=0}^{\infty} \sigma_{j, k}^m \, \epsilon^{k} \ \textrm{ and } \
v_{j}^m(\epsilon) = \sum_{k=0}^{\infty} v_{j, k}^m \, \epsilon^{k},
\end{equation}
for $m \in \{1,2\}$, and how to compute the coefficients of these series.

\begin{lemma}\label{lemma-coefficients}
	Let $\mathcal{A},\mathcal{B} \in \C^{d\times d}$ and assume that $\sigma_0$ is a simple eigenvalue of $\mathcal{A}$, with a corresponding eigenvector $v_0$. Then there is $\epsilon_0>0$ such that, for $|\epsilon|<\epsilon_0$, there is a simple eigenvalue $\sigma(\epsilon)$  of $\mathcal{A}_\epsilon = \mathcal{A}+\epsilon\mathcal{B}$, and a corresponding eigenvector $v(\epsilon)$ which can be written in the form
	\begin{equation}\label{exp-eigenv-1}
	\sigma(\epsilon) = \sum_{k=0}^{\infty} \sigma_k \epsilon^k \ \textrm{ and } \  v(\epsilon) = \sum_{k=0}^{\infty} v_k \epsilon^k,
	\end{equation}
	satisfying the condition $\left\langle  v(\epsilon), v_0^* \right\rangle  = 1,$ where $v_0^*$ is the eigenvector of $A^*$ associated to the eigenvalue $\bar{\sigma}_0$.
\end{lemma}

From  Kato, for $\epsilon$ small, we have a simple eigenvalue $\sigma(\epsilon)$ and we can choose a corresponding eigenvector $v(\epsilon)$ depending holomorphically of $\epsilon$ such that $v(0) = v_0$.

Now, to find explicit formulas for the coefficients, by expanding in powers of $\epsilon$, we get for the coefficient of $\epsilon^k (k \geqslant 1)$ the identities
$$ (\mathcal{A}-\sigma_0)v_k = -\left(\mathcal{B}v_{k-1}-\sigma_{k}v_0- \sum_{\ell =1}^{k-1} \sigma_{k-\ell}v_\ell \right)$$
Assuming that we have found the pair $(\sigma_\ell,v_\ell)$, for $\ell < k$, we now determine the pair $(\sigma_k,v_k)$. First $\sigma_k$ is determined by the condition that $(\mathcal{A}-\sigma_0)v_k$ is orthogonal to $\ker(\mathcal{A}^* - \bar{\sigma}_0)$.

It is clear that $\bar{\sigma}_0$ is a simple eigenvalue of $\mathcal{A}^*$. Moreover by A. Aslanyan and E. Davies \cite{Asl-Dav} (see Lemma 1 and the proof of Theorem 1) the corresponding eigenvector $v_0^*$ satisfies $\langle v_0, v_0^* \rangle \neq 0$ and we can normalize $v_0^*$ by
$$ \langle v_0, v_0^* \rangle= 1.$$
Hence
$$ \sigma_k = \langle \mathcal{B}v_{k-1}, v_0^* \rangle - \sigma_{k}v_0- \sum_{\ell =1}^{k-1} \sigma_{k-\ell}\langle v_\ell, v_0^* \rangle = \langle \mathcal{B}v_{k-1}, v_0^* \rangle,$$
and then there exists a unique $v_k$ such that $ \langle v_k, v_0^* \rangle= 0.$

As considered in \cite{Asl-Dav}, the Lemma above is true under much weaker assumption on $\mathcal{A}$ (a closed operator) and $\mathcal{B}$, where $\C^k$ is replaced by an Hilbert space $\mathcal{H}$.

\begin{remark}\label{after-lemma-coef}
	If $\mathcal{A}$ is normal then $A^*u_0$ satisfies
	$ \mathcal{A}\mathcal{A}^*u_0 = \mathcal{A}^*\mathcal{A}u_0 = \sigma_0 \mathcal{A}^*u_0.$
	Hence $\mathcal{A}^*u_0 = \lambda u_0,$ for some $\lambda.$ Taking the scalar product with $u_0$ we obtain $\lambda = \bar{\sigma}_0$ and $u_0^*=u_0$. Therefore the coefficients in \eqref{exp-eigenv-1} satisfy
	\begin{equation}\label{exp-eigenv-3}
	\sigma_k = \left\langle  {\mathcal{B}} v_{k-1} , v_0 \right\rangle \mbox{ and } \
	({\mathcal{A}} - \sigma_0 I) v_k =  \sum_{n=1}^{k}\sigma_n v_{k- n}  - {\mathcal{B}} v_{k-1},
	\end{equation}
	for $k \in \N.$
\end{remark}

Since the eigenvalues of $Q_j = \omega \mathcal{D}_j$ are simple, for all $j$, we can use Lemma \ref{lemma-coefficients} and Remark \ref{after-lemma-coef} to calculate the eigenvalues of $Q_{j}(\epsilon)$.
It follows from condition $\left\langle  v_{j, \epsilon}^m , v_{j, 0}^m  \right\rangle  = 1$ that the eigenvectors $v_{j, k}^m$ are of type
\begin{equation}
v_{j, k}^1 = (0, \alpha_{j, k}) \ \textrm{ and } \ v_{j, k}^2 = (\beta_{j, k}, 0),
\end{equation}
for some complex numbers $\alpha_{j, k}$ and $\beta_{j, k}$, and
\begin{equation*}
(Q_{j} - \sigma_{j, 0}^m I) v_{j, k}^m =  \sum_{n=1}^{k}\sigma_{j, n}^m  v _{j, k-n}^m  - \mathcal{R}_{j} v_{j, k-1}^m,
\end{equation*}
for all $k \in \N$.

Thus, we obtain $\sigma_{j, 1}^1 = a_j$, $\sigma_{j, 1}^2 = d_j$,
\begin{equation}\label{sigma-calc}
\sigma_{j, k}^1 = \left\langle  \mathcal{R}_{j} v_{j, k-1}^1 , v_{j, 0}^1 \right\rangle = b_{j} \alpha_{j, k-1},  \ \forall k\geqslant 2,
\end{equation}
and
\begin{equation}\label{sigma-calc2}
\sigma_{j, k}^2 = \left\langle  \mathcal{R}_{j} v_{j, k-1}^2 , v_{j, 0}^2 \right\rangle = c_{j} \beta_{j, k-1},  \ \forall k \geqslant  2.
\end{equation}

By induction on $k$ it can be proved that
\begin{equation*}
\alpha_{j, k} =(2 \omega j)^{-1} \left( \alpha_{j, k-1} (a_{j} - d_{j})  + \sum_{n=2}^{k-1}  b_{j}\alpha_{j, n-1} \alpha_{j, k- n}\right),
\end{equation*}
and
\begin{equation*}
\beta_{j, k} =(2 \omega j)^{-1} \left( \beta_{j, k-1} (a_{j} - d_{j})  + \sum_{n=2}^{k-1}  c_{j}\beta_{j, n-1} \beta_{j, k- n}\right),
\end{equation*}

\noindent
whenever $k \geqslant 3$, and the next table exhibits the first three terms.

\vspace{2mm} \noindent
\begin{minipage}{\textwidth}
	\centering
	\renewcommand{\arraystretch}{2}
	\begin{small}
		\begin{tabular}{c|c|c|c|c|}
			$k$  & $\sigma_{j, k}^1$  & $\alpha_{j, k}$  & $\sigma_{j,k}^2$                            & $\beta_{j, k}$  \\  
			\hline 
			1  & $a_{j}$  &  $- \dfrac{c_{j}}{ 2 \omega j}$  & $d_{j}$ 	& $\dfrac{b_{j}}{ 2 \omega j}$   \\ [5pt]   
			2  &  $- \dfrac{b_{j}c_{j}}{2 \omega j}$ & $\dfrac{c_{j}(a_{j} - d_{j})}{ 4 \omega^2 j^2}$  & $\dfrac{b_{j}c_{j}}{2 \omega j}$ & $\dfrac{b_{j}(d_{j} - a_{j})}{ 4 \omega^2 j^2}$ \\[5pt]  
			3  & $\dfrac{b_{j}c_{j}(d_{j} - a_{j})}{ 4 \omega^2 j^2}$ & $\dfrac{b_{j}c_{j}^2 - c_{j}(d_{j} - a_{j})^2}{ 8 \omega^3 j^3}$  & $\dfrac{b_{j}c_{j}(d_{j} - a_{j})}{ 4 \omega^2 j^2}$ & $\dfrac{b_{j}^2c_{j} + b_{j}(d_{j} - a_{j})^2}{ 8 \omega^3 j^3}$ \\[5pt]  
		\end{tabular} \vspace{3mm}
	\end{small}	
\end{minipage}

\begin{proposition}
	There exists $\epsilon_0>0$ such that the operator $Q(\epsilon) = Q + \epsilon \mathcal{R}$ is strongly diagonalizable, for $|\epsilon|<\epsilon_0$.
\end{proposition}

\begin{proof}
	The eigenvalues and eigenvectors of $Q_{j}(\epsilon) = Q_j + \epsilon \mathcal{R}_j$ can be written in the form \eqref{exp-eigenv-2}, for $j \in\N$. Moreover, for each $k\in\N$, we have
	\begin{equation}\label{order-ctrl}
	|\alpha_{j, k} | =   \mathcal{O}( j^{ k(\delta - 1)}) \ \textrm{ and } \ |\beta_{j, k} |  =   \mathcal{O}( j^{k(\delta - 1)}),
	\end{equation}
	as $j \to \infty$,  and that the matrices
	\begin{equation*}
	S_{j}(\epsilon) = [v_{j}^1(\epsilon) \ \ v_{j}^2(\epsilon)] =
	\left [
	\begin{array}{cc}
	1                &   \displaystyle\sum_{k=1}^{\infty} \beta_{j, k}\epsilon^k   \\
	\displaystyle \sum_{k=1}^{\infty} \alpha_{j, k}\epsilon^k   &  1
	\end{array}
	\right]
	\end{equation*}
	satisfy $Q_{j}(\epsilon) = S_{j}(\epsilon) \, \mbox{diag} ( \sigma^1{\ell}(\epsilon) \, , \, \sigma^2{\ell}(\epsilon)  ) \, S_{j}(\epsilon)^{-1}$.
	
	Since $\delta < 1$, it follows from \eqref{order-ctrl} that there is $\epsilon_0 >0$, and constants $C_1, C_2$ and $j_0$ such that, for $j \geqslant j_0$ and $|\epsilon| \leqslant \epsilon_0$, we have
	\begin{equation}\label{estim-eigenv}
	\left |\sum_{k=1}^{\infty} \alpha_{j, k}\epsilon^k  \right|  \leqslant  C_1 \dfrac{|\epsilon|}{1 -|\epsilon|}
	\ \textrm{ and } \
	\left |\sum_{k=1}^{\infty} \beta_{j, k}\epsilon^k  \right|   \leqslant  C_2 \dfrac{|\epsilon|}{1 -|\epsilon|}.
	\end{equation}
	
	This shows that $\sup_{j \in \N} \|S_{j}(\epsilon)\| < \infty$, for any $\epsilon \leqslant \epsilon_0$.
	
	On the other hand, for $S_{j}(\epsilon)^{-1}$ we have
	\begin{equation*}
	\|S_{j}(\epsilon)^{-1}\| = |\det (S_{j}(\epsilon))|^{-1} \|S_{j}(\epsilon)\| =  \left|1 -  \sum_{k=1}^{\infty} \alpha_{j, k}\epsilon^k \sum_{k=1}^{\infty} \beta_{j, k}\epsilon^k \right|^{-1} \|S_{j}(\epsilon)\|.
	\end{equation*}
	
	Now, from \eqref{estim-eigenv}, for $j\geqslant j_0$ and $|\epsilon| < \min \{\epsilon_0, (1+\sqrt{2C_1C_2})^{-1}\}$,  we have
	
	\begin{equation*}
	\left|1 -  \left(\sum_{k=1}^{\infty} \alpha_{j, k}\epsilon^k\right) \left(\sum_{k=1}^{\infty} \beta_{j, k}\epsilon^k\right) \right|
	= 1 - C_1C_2 \dfrac{|\epsilon|^2}{(1 -|\epsilon|)^2} \geqslant \frac{1}{2}.
	\end{equation*}
	
	So there is $\epsilon'>0$, $C>0$ and $j_0 \in \N$ such that,
	$$\|S_{j}(\epsilon)^{-1}\| \leqslant C, \mbox{ for } j\geqslant j_0 \mbox{ and } |\epsilon|<\epsilon'.$$
	
	Therefore $Q(\epsilon)$ is strongly diagonalizable.
	
\end{proof}

\begin{example}
	Perturbations by operators with matrix symbol given by
	
	\begin{equation*}
	{\mathcal{R}}_0 = 0, \ \textrm{ and } \
	{\mathcal{R}}_j = \left [
	\begin{array}{cc}
	0   & \gamma_j \\
	\gamma_j  & 0
	\end{array}
	\right], \ j \in \N,
	\end{equation*}
	
	\noindent where $|\gamma_j| = \mathcal{O}(j^{\delta})$, for some $0 \leqslant \delta < 1$, were studied in Example \ref{n-comm-example1}

	By using the explicit formulas for the coefficients above, we obtain the first terms of the expansion of eigenvalues
	\begin{equation*}\label{sigma_j_eps_1}
	\sigma_{j}^1(\epsilon) = - \omega j - \dfrac{\gamma_j^2}{2 \omega j} \, \epsilon^2 + \dfrac{\gamma_j^4}{(2 \omega j)^3} \, \epsilon^4
	+ \mathcal{O}(\epsilon^6),
	\end{equation*}
	and
	\begin{equation*}\label{sigma_j_eps_2}
	\sigma_{j}^2(\epsilon) = \omega j + \dfrac{\gamma_j^2}{2 \omega j} \, \epsilon^2 - \dfrac{\gamma_j^4}{(2 \omega j)^3} \, \epsilon^4
	+ \mathcal{O}(\epsilon^6).
	\end{equation*}
	
	In particular, all the odd powers of $\epsilon$ in these expansions are null.
	
	Observe that the eigenvalues of the matrices $Q_{j}(\epsilon) = \omega \mathcal{D}_j + \epsilon \mathcal{R}_j$ are 
	$$
	\sigma_{j}(\epsilon) = \sqrt[2]{\omega^2 j^2 +  \epsilon^2 \gamma_j^2}, \ j \in \N_0,
	$$	
	therefore the expansion of $\sigma_{\ell}(\epsilon)$ in powers of $\epsilon$ is given by
	\begin{equation}\label{expansion-sqrt}
	\sigma_j^m(\epsilon) = (-1)^m\omega j + \sum_{k=1}^{+\infty}(-1)^{m(2k-1)} \dfrac{\gamma_{j}^{2k}}{(\omega j)^{2k-1}} a_k\epsilon^{2k},
	\end{equation}
	where $m=1,2$ and
	\begin{align*}
	a_k \doteq \binom{1/2}{k} 
	& = \dfrac{(-1)^{k-1}(2k - 2)!}{2^{2k-1}k! (k-1)!}
	\end{align*}	
	satisfies $|a_k| \leqslant 1/2$, for $k \in \N$.
	
	\vspace{4mm}
	\noindent {\bf Claim:} There exist  $\eta,\epsilon_0>0$ and $j_0, N \in\N$ such that
	\begin{equation}\label{high-terms}
	\sum_{k\geqslant N+1}\left| \dfrac{\gamma_{j}^{2k}}{(\omega j)^{2k-1}} a_k\epsilon^{2k}\right| 
	\leqslant j^{-\eta},
	\end{equation}
	for $j\geqslant j_0$ and $|\epsilon|<\epsilon_0.$
	
	\begin{proof}
		Since $|\gamma_j| = \mathcal{O}(j^{\delta})$, with $0 \leqslant \delta < 1$, there are $j_0\in \N$ and $C>0$ such that
		$$
		|\gamma_j| \leqslant C j^\delta, \mbox{ for } j\geqslant j_0.
		$$ 
		
		Now we can choose $N\in\N$ such that $2N(\delta - 1) + 1 < 0,$ and set $$\eta = -[2N(\delta - 1) + 1].$$
		
		Thus, if $k\geqslant N$ and $j\geqslant j_0$ then
		\begin{equation*}
		\left| \dfrac{\gamma_{j}^{2k}}{(\omega j)^{2k-1}} a_k\right| 
		\leqslant \dfrac{C^{2k}j^{2k\delta}}{2|\omega|^{2k-1} j^{2k-1}} \leqslant \dfrac{|\omega|}{2} \left(\dfrac{C}{|\omega|}\right)^{2k} j^{-\eta}.
		\end{equation*}

		Therefore, setting $\epsilon_0=|\omega|/2C$, for any $|\epsilon| <\epsilon_0$ we have
		\begin{align*}
		\sum_{k\geqslant N+1}\left| \dfrac{\gamma_{j}^{2k}}{(\omega j)^{2k-1}} a_k\epsilon^{2k}\right| 
		& \leqslant \frac{|\omega|}{2}j^{-\eta} \sum_{k\geqslant N+1}\left(\dfrac{C}{|\omega|} |\epsilon|\right)^{2k} \\
		& \leqslant \frac{|\omega|}{2}j^{-\eta} \sum_{k\geqslant N+1}\left(\dfrac{1}{2}\right)^{2k} = \\
		& \frac{|\omega|}{3\cdot 2^{2N+1}}j^{-\eta}.
		\end{align*}
		
		Increasing $N$, if necessary, we obtain \eqref{high-terms}
	\end{proof}
	
	Going back to our example  and keeping in mind the last claim, we have
	\begin{align*}
	\left|\tau + \sigma_j^m(\epsilon) \right|  & = \left|\tau + (-1)^m\omega j + \sum_{k=1}^{+\infty}\dfrac{(-1)^{m(2k-1)} \gamma_{j}^{2k}}{(\omega j)^{2k-1}} a_k\epsilon^{2k} \right| \\ 
	& \geqslant \left|\tau + (-1)^m\omega j + \sum_{k=1}^{N} \dfrac{(-1)^{m(2k-1)}\gamma_{j}^{2k}} {(\omega j)^{2k-1}} a_k\epsilon^{2k} \right| - j^{-\eta},
	\end{align*}
	for $j\geqslant j_0$ and $|\epsilon|<\epsilon_0.$
	
	Therefore, if there are constants $C_1,\theta_1,\epsilon_1>0$ and $j_1,N\in\N$ such that $\theta_1 < \eta$ and
	$$
	\left|\tau + (-1)^m\omega j + \sum_{k=1}^{N} \dfrac{(-1)^{m(2k-1)}\gamma_{j}^{2k}} {(\omega j)^{2k-1}} a_k\epsilon^{2k} \right| \geqslant C_1 j^{-\theta_1},
	$$
	for $j\geqslant j_1$ and $|\epsilon|<\epsilon_1,$ then $L(\epsilon) = D_t + \omega D_x + \epsilon \mathcal{R}$ is (GH).
\end{example}

\medskip
\begin{example}
	Taking $\delta < 1/2, \  N = 1$ and $\eta = 1-2\delta$ in the last example, after a simple adjustment in the proof, we  have the following: if 
	$L=D_t + \omega D_x$ is (GH) with
	\begin{equation}
	\inf_{\tau \in \Z}\left|\tau + \omega \ell \right| \geqslant C_2 \ell^{-\theta},
	\end{equation} 
	for some $0 < \theta < 1-2\delta$, then $L(\epsilon) = D_t + \omega D_x + \epsilon \mathcal{R}$ is (GH).
\end{example}

\medskip
\begin{example}
	Consider the $E-$invariant operator $\mathcal{R}$ with matrix symbol given by
	\begin{equation*}
	{\mathcal{R}}_0 = 0 \in \C,
	\textrm{ and } \
	{\mathcal{R}}_j = \left [
	\begin{array}{cc}
	a_j  & 0 \\[2mm]
	c_j  & d_j
	\end{array} \right] \in \C^{2\times 2},  \ j \in \N.
	\end{equation*}
	
	In this case we have $\sigma_{j, k}^1 = \sigma_{j, k}^2 = 0$, for all $k \geqslant 2$ and therefore
	\begin{equation*}
	\sigma^1_{j}(\epsilon) = - \omega j + a_j \epsilon \ \textrm{ and } \ \sigma^2_{j}(\epsilon) =  \omega j + d_j \epsilon,
	\end{equation*}
	for all $j\in\N.$
	
	Thus $L_{\epsilon}$ is globally hypoelliptic if, and only if,
	\begin{equation}\label{jordan-block-cond}
	\inf_{\tau \in \Z} \left|\tau - \omega \ell + a_\ell \epsilon \right| \geqslant C \ell^{-\theta} \ \textrm{ and } \
	\inf_{\tau \in \Z} \left|\tau + \omega \ell + d_\ell \epsilon \right| \geqslant C \ell^{-\theta},
	\end{equation}
	are verified for some $C,\theta>0$ and all $\ell \geqslant\ell_0.$
	
	Notice that the entries $c_j$ of the matrices ${\mathcal{R}}_j$ play no role in the study of global hypoellipticity. 
	
	In particular, if the matrices $\mathcal{R}_j$ are nilpotent then $a_j=d_j=0$, for all $j\in\N$, and $L(\epsilon) = D_t + \omega D_x + \epsilon \mathcal{R}$ is (GH) if and only if $L= D_t + \omega D_x$ is (GH).
\end{example}


\begin{thebibliography}{10}
	
	\bibitem{Asl-Dav}
	A.~{Aslanyan} and E.~{Davies}.
	\newblock {Spectral instability for some Schr\"odinger operators.}
	\newblock {\em {Numer. Math.}}, 85(4):525--552, 2000.
	
	\smallskip
	\smallskip
	\bibitem{BERG94}
	A.~P. {Bergamasco}.
	\newblock {Perturbations of globally hypoelliptic operators.}
	\newblock {\em {J. Differ. Equations}}, 114(2):513--526, 1994.
	
	\smallskip
	\bibitem{BERG99}
	A.~P. {Bergamasco}.
	\newblock {Remarks about global analytic hypoellipticity.}
	\newblock {\em {Trans. Am. Math. Soc.}}, 351(10):4113--4126, 1999.
	
	\smallskip
	\bibitem{BCM}
	A.~P. {Bergamasco}, P.~D. {Cordaro}, and P.~A. {Malagutti}.
	\newblock {Globally hypoelliptic systems of vector fields.}
	\newblock {\em {J. Funct. Anal.}}, 114(2):267--285, 1993.
	
	\smallskip
	\bibitem{BDGK15}
	A.~P. {Bergamasco}, P.~L. {Dattori da Silva}, R.~B. {Gonzalez}, and
	A.~{Kirilov}.
	\newblock {Global solvability and global hypoellipticity for a class of complex
		vector fields on the 3-torus.}
	\newblock {\em {J. Pseudo-Differ. Oper. Appl.}}, 6(3):341--360, 2015.
	
	\smallskip
	\bibitem{CC00}
	W.~{Chen} and M.~{Chi}.
	\newblock {Hypoelliptic vector fields and almost periodic motions on the torus
		$T^n$.}
	\newblock {\em {Commun. Partial Differ. Equations}}, 25(1-2):337--354, 2000.
	
	\smallskip
	\bibitem{DasRuz16}
	A.~{Dasgupta} and M.~{Ruzhansky}.
	\newblock {Eigenfunction expansions of ultradifferentiable functions and
		ultradistributions.}
	\newblock {\em {Trans. Am. Math. Soc.}}, 368(12):8481--8498, 2016.
	
	\smallskip
	\bibitem{Day-98}
	J.~{Day}, W.~{So}, and R.~C. {Thompson}.
	\newblock {The spectrum of a Hermitian matrix sum.}
	\newblock {\em {Linear Algebra Appl.}}, 280(2-3):289--332, 1998.
	
	\smallskip
	\bibitem{AvGrKi}
	F.~de~{\'Avila Silva}, A.~{Kirilov}, and T.~{Gramchev}.
	\newblock Global hypoellipticity for first-order operators on closed smooth
	manifolds.
	\newblock {\em {J. Anal. Math. (to appear) {arXiv:1507.08880[math.AP]}}}, 2015.
	
	\smallskip
	\bibitem{DR14}
	J.~{Delgado} and M.~{Ruzhansky}.
	\newblock {Crit\`eres portant sur des symboles et noyaux pour les classes de
		Schatten et $r$-nucl\'earit\'e sur les vari\'et\'es compactes.}
	\newblock {\em {C. R., Math., Acad. Sci. Paris}}, 352(10):779--784, 2014.
	
	\smallskip
	\bibitem{DR-Fmul}
	J.~{Delgado} and M.~{Ruzhansky}.
	\newblock Fourier multipliers, symbols and nuclearity on compact manifolds.
	\newblock {\em {J. Anal. Math. (to appear) {arXiv:1404.6479[math.AP]}}}, 2014.
	
	\smallskip
	\bibitem{DR14JFA}
	J.~{Delgado} and M.~{Ruzhansky}.
	\newblock {Schatten classes on compact manifolds: kernel conditions.}
	\newblock {\em {J. Funct. Anal.}}, 267(3):772--798, 2014.
	
	\smallskip
	\bibitem{DelgRuzTok17}
	J.~Delgado, M.~Ruzhansky, and N.~Tokmagambetov.
	\newblock Schatten classes, nuclearity and nonharmonic analysis on compact
	manifolds with boundary.
	\newblock {\em Journal de Mathématiques Pures et Appliquées}, 107(6):758 --
	783, 2017.
	
	\smallskip
	\bibitem{GPR}
	T.~{Gramchev}, S.~{Pilipovic}, and L.~{Rodino}.
	\newblock {Eigenfunction expansions in $\mathbb R^{n}$.}
	\newblock {\em {Proc. Am. Math. Soc.}}, 139(12):4361--4368, 2011.
	
	\smallskip
	\bibitem{GPY1}
	T.~{Gramchev}, P.~{Popivanov}, and M.~{Yoshino}.
	\newblock {Global solvability and hypoellipticity on the torus for a class of
		differential operators with variable coefficients.}
	\newblock {\em {Proc. Japan Acad., Ser. A}}, 68(3):53--57, 1992.
	
	\smallskip
	\bibitem{GW1}
	S.~J. {Greenfield} and N.~R. {Wallach}.
	\newblock {Global hypoellipticity and Liouville numbers.}
	\newblock {\em {Proc. Am. Math. Soc.}}, 31:112--114, 1972.
	
	\smallskip
	\bibitem{GW3}
	S.~J. {Greenfield} and N.~R. {Wallach}.
	\newblock {Globally hypoelliptic vector fields.}
	\newblock {\em {Topology}}, 12:247--253, 1973.
	
	\smallskip
	\bibitem{GW2}
	S.~J. {Greenfield} and N.~R. {Wallach}.
	\newblock {Remarks on global hypoellipticity.}
	\newblock {\em {Trans. Am. Math. Soc.}}, 183:153--164, 1973.
	
	\smallskip
	\bibitem{Hof}
	K.~{Hoffman} and R.~{Kunze}.
	\newblock {\em {Linear algebra}}.
	\newblock Prentice-Hall, 1st ed. edition, 1971.
	
	\smallskip
	\bibitem{Hou79}
	J.~{Hounie}.
	\newblock {Globally hypoelliptic and globally solvable first order evolution
		equations.}
	\newblock {\em {Trans. Am. Math. Soc.}}, 252:233--248, 1979.
	
	\smallskip
	\bibitem{HOU82}
	J.~{Hounie}.
	\newblock {Globally hypoelliptic vector fields on compact surfaces.}
	\newblock {\em {Commun. Partial Differ. Equations}}, 7:343--370, 1982.
	
	\smallskip
	\bibitem{Jamison54}
	S.~L. Jamison.
	\newblock Perturbation of normal operators.
	\newblock {\em Proc. Amer. Math. Soc.}, 5:103--110, 1954.
	
	\smallskip
	\bibitem{kato}
	T.~{Kato}.
	\newblock {\em {Perturbation theory for linear operators. Reprint of the corr.
			print. of the 2nd ed. 1980.}}
	\newblock Berlin: Springer-Verlag, reprint of the corr. print. of the 2nd ed.
	1980 edition, 1995.
	
	\smallskip
	\bibitem{Klya-98}
	A.~A. {Klyachko}.
	\newblock {Stable bundles, representation theory and Hermitian operators.}
	\newblock {\em {Sel. Math., New Ser.}}, 4(3):419--445, 1998.
	
	\smallskip
	\bibitem{Petr06}
	G.~{Petronilho}.
	\newblock {Simultaneous reduction of a family of commuting real vector fields
		and global hypoellipticity.}
	\newblock {\em {Isr. J. Math.}}, 155:81--92, 2006.
	
	\smallskip
	\bibitem{Petr11}
	G.~{Petronilho}.
	\newblock {Global hypoellipticity, global solvability and normal form for a
		class of real vector fields on a torus and application.}
	\newblock {\em {Trans. Am. Math. Soc.}}, 363(12):6337--6349, 2011.
	
	\smallskip
	\bibitem{Rellich}
	F.~{Rellich}.
	\newblock {Perturbation theory of eigenvalue problems.}
	\newblock {Notes on Mathematics and its Applications. New York-London-Paris:
		Gordon and Breach Science Publishers. X, 127 p. (1969).}, 1969.
	
	\smallskip
	\bibitem{RuzTok16IMRN}
	M.~Ruzhansky and N.~Tokmagambetov.
	\newblock Nonharmonic analysis of boundary value problems.
	\newblock {\em International Mathematics Research Notices},
	2016(12):3548--3615, 2016.
	
	\smallskip
	\bibitem{schmidt}
	W.~M. Schmidt.
	\newblock {\em Diophantine approximation}, volume 785.
	\newblock Springer Science \& Business Media, 1996.
	
	\smallskip
	\bibitem{Shubin}
	M.~{Shubin}.
	\newblock {\em {Pseudodifferential operators and spectral theory. Transl. from
			the Russian by Stig I. Andersson. 2nd ed.}}
	\newblock Berlin: Springer, 2nd ed. edition, 2001.
	
	\smallskip
	\bibitem{Tomp-76}
	R.~{Thompson}.
	\newblock {The behavior of eigenvalues and singular values under perturbations
		of restricted rank.}
	\newblock {\em {Linear Algebra Appl.}}, 13:69--78, 1976.
	
	\smallskip
	\bibitem{Tomp-71}
	R.~C. {Thompson} and L.~J. {Freede}.
	\newblock {On the eigenvalues of sums of Hermitian matrices.}
	\newblock {\em {Linear Algebra Appl.}}, 4:369--376, 1971.
	
	\smallskip
	\bibitem{Helm-55}
	H.~{Wielandt}.
	\newblock {On eigenvalues of sums of normal matrices.}
	\newblock {\em {Pac. J. Math.}}, 5:633--638, 1955.
	
\end{thebibliography}
\end{document}